\documentclass[10pt]{amsart}
\usepackage{epsfig}
\usepackage{amsmath,amssymb,amsfonts,amsthm,
latexsym, amscd}
\usepackage{latexsym}
\usepackage{amssymb}
\usepackage{eucal}%    caligraphic-euler fonts: \mathcal{ }
\usepackage{eufrak}%   frak-euler        fonts: \mathfrak{ }
\usepackage[all]{xypic}
\usepackage{xspace}
\usepackage{mathrsfs}
\theoremstyle{plain}
\newtheorem{theorem}{Theorem}[section]

\newtheorem{proposition}[theorem]{Proposition}

\theoremstyle{definition}
\newtheorem{definition}{Definition}

\theoremstyle{remark}
\newtheorem{remark}{Remark}

\numberwithin{equation}{section}
\begin{document}

%%%
%%%
%%%%%%%%%%%%%%%%%%%%%%%%%%%%%%%%%%%%%%%%%%%%%%%%%%%%%%%%%%%%%%%%%%%%%%%%%%
%%
%%%
\title[Folding $3$-noncrossing RNA pseudoknot structures]
      {Folding $3$-noncrossing RNA pseudoknot structures}
\author{Fenix W.D. Huang, Wade W.J. Peng and Christian M. Reidys$^{\,\star}$}
\address{Center for Combinatorics, LPMC-TJKLC %XXX%
           \\
         Nankai University  \\
         Tianjin 300071\\
         P.R.~China\\
         Phone: *86-22-2350-6800\\
         Fax:   *86-22-2350-9272}
\email{reidys@nankai.edu.cn}%XXXX
\thanks{}
\date{September, 2008}
\begin{abstract}
In this paper we present a selfcontained analysis and description of 
the novel {\it ab initio} folding algorithm {\sf cross}, which generates the 
minimum free energy (mfe), $3$-noncrossing, $\sigma$-canonical RNA structure. 
Here an RNA structure is $3$-noncrossing if it does not contain more than 
three mutually crossing arcs and $\sigma$-canonical, if each of its stacks 
has size greater or equal than $\sigma$. Our notion of mfe-structure is 
based on a specific concept of pseudoknots and respective loop-based 
energy parameters.
The algorithm decomposes into three parts: the first is the
inductive construction of motifs and shadows, the 
second is the generation of the skeleta-trees rooted in irreducible 
shadows and the third is the saturation of skeleta via context 
dependent dynamic programming routines.
\end{abstract}
\maketitle
{{\small
%\tableofcontents
}}

%%%
%%%%%%%%%%%%%%%%%%%%%%%%%%%%%%%%%%%%%%%%%%%%%%%%%%%%%%%%%%%%%%%%%%%%%%%%%
%%%

\section{Introduction and background}\label{S:Introduction}

%%%
%%%%%%%%%%%%%%%%%%%%%%%%%%%%%%%%%%%%%%%%%%%%%%%%%%%%%%%%%%%%%%%%%%%%%%%%%
%%%

In this paper we introduce the {\it ab initio} folding algorithm cross
which folds RNA (ribonucleic acid) sequences \cite{Tinoco:73} into 
pseudoknot structures. 
We give a selfcontained presentation and analysis of {\it cross}, whose
source code is publicly available at
\begin{equation*}
{\tt www.combinatorics.cn/cbpc/cross.html}
\end{equation*}
Supplementary material, such as detailed description of the loop-energies and
all implementation details can be found at the above web-site.
Let us begin by providing some background on RNA sequences and structures.
An RNA molecule is firstly described by its primary sequence, a linear string
composed by the four nucleotides {\bf A}, {\bf G}, {\bf U} and {\bf C}
together with the Watson-Crick ({\bf A-U}, {\bf G-C}) and ({\bf U-G}) base
pairing rules. Secondly, RNA, structurally less constrained than its chemical
relative DNA, folds into helical structures by pairing the nucleotides and
thereby lowering their minimum free energy, see Fig.\ref{F:tRNA}
%%%
%%%%%%%%%%%%%%%%%%%%%%%%%%%%%%%%%%%%%%%%%%%%%%%%%%%%%%%%%%%%%%%%%%%%%%%%%
%%%
\begin{figure}[ht]
\centerline{\epsfig{file=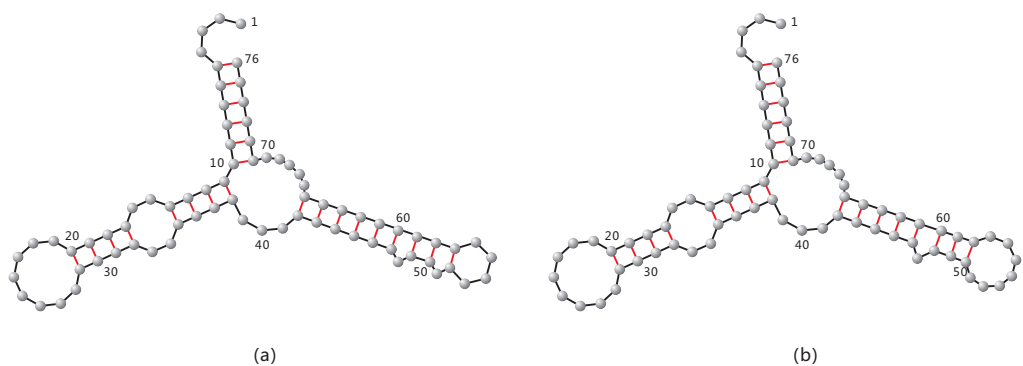,width=0.9\textwidth} \hskip8pt}
\caption{\small The phenylalanine tRNA (re)visited: (a) represents the 
structure
of phenylalanine tRNA, as folded by ViennaRNA
\cite{Vienna:Server,ViennaRNA}. 
(b) shows the phenylalanine structure as folded by {\sf cross} with minimum 
stack size $3$. Note that {\sf cross} does not contain any stack which
 size $\le 3$, therefore (b) is different from (a) slightly 
in $48$ to $60$.
}\label{F:tRNA}
\end{figure}
%%%
%%%%%%%%%%%%%%%%%%%%%%%%%%%%%%%%%%%%%%%%%%%%%%%%%%%%%%%%%%%%%%%%%%%%%%%%%
%%%
Accordingly, RNA exhibits a variety of 3-dimensional structural
configurations, the so called tertiary structures, determining the 
functionality of the molecule.
Besides the noncrossing base pairings found in RNA secondary structures
there exist further types of nucleotide interactions \cite{Westhof:92a}. 
These bonds are called pseudoknots and occur in functional RNA like for 
instance RNAseP \cite{Loria:96a} as well as ribosomal RNA \cite{Konings:95a}.
Indeed, RNA exhibits a diversity of biochemical capabilities
\cite{Science:05a}, proved by the discovery of catalytic RNAs, or 
ribozymes \cite{Loria:96a}, in 1981. 
Like proteins, RNA is capable of catalyzing reactions whereas transfer RNA
acts as a messenger between DNA and protein.
%%%
%%%%%%%%%%%%%%%%%%%%%%%%%%%%%%%%%%%%%%%%%%%%%%%%%%%%%%%%%%%%%%%%%%%%%%%%%
%%%
\begin{figure}[ht]
\centerline{\epsfig{file=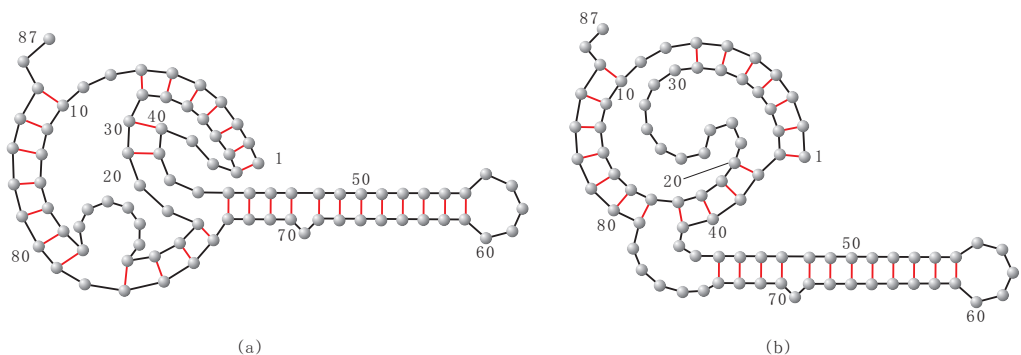,width=0.9\textwidth} \hskip8pt}
\caption{\small The HDV-pseudoknot structure: (a) displays the structure
as folded by Rivas and Eddy's algorithm \cite{RE:98}. (b) shows the structure
as folded by {\sf cross} with minimum stack size $3$. }\label{F:HDV}
\end{figure}
%%%
%%%%%%%%%%%%%%%%%%%%%%%%%%%%%%%%%%%%%%%%%%%%%%%%%%%%%%%%%%%%%%%%%%%%%%%%%
%%%

In light of these RNA functionalities the question of RNA structure
prediction appears to be of relevance. The first mfe-folding algorithms 
for RNA secondary structure are due to 
\cite{Jaces:1960,Tinoco:1971,DeLisi:1971} and the first DP folding routines 
for secondary structures were given by Waterman {\it et al.}
\cite{Waterman:78a,Waterman:86,Zuker:1981,Nussinov:1980}, 
predicting the loop-based
mfe-secondary structure \cite{Tinoco:73} in $O(n^3)$-time and $O(n^2)$-space.
The general problem of RNA
structure prediction under the widely used thermodynamic model is
known to be NP-complete when the structures considered include
arbitrary pseudoknots \cite{Lyngso}. 
There exist however, polynomial time 
folding algorithms, capable of the energy based prediction of certain 
pseudoknots: Rivas et.al. \cite{RE:98}, Uemura et.al. \cite{Uemura:99a}, 
Akutsu \cite{Akutsu:00a} and Lyngs\o \cite{Lyngso}. In the following we shall
use the term pseudoknot synonymous with cross-serial dependencies between 
pairs of nucleotides \cite{Searls:02,Cao:06}. As for the
{\it ab initio} folding of pseudoknot RNA, we find the following two
paradigms: Rivas and Eddy's \cite{RE:98} gap-matrix variant of Waterman's
DP-folding routine for secondary structures
\cite{Waterman:78a,Waterman:79a,Waterman:80,Waterman:86,Nussinov:1980}, 
maximum weighted matching algorithms \cite{MWM:65,MWM:76} and the latter 
taylored for pseudoknot prediction \cite{MWM:95,MWM:98}.
The former method folds into a somewhat ``mysterious'' class of
pseudoknots \cite{RE:00} in polynomial time. Algorithms along these
lines have been developed by Dirks and Pierce {\cite{Dirks:04}}, 
Reeder and Giegerich {\cite{Reeder:04}} and Ren {\it et
al.}~{\cite{Ren:05}}. Additional ideas for pseudoknot folding involve the 
iterated loop matching approach \cite{ILM} and the sampling of RNA structures 
via the Markov-chain Monte-Carlo method \cite{Nebel:06}.

Let us now have a closer look at the DP-paradim by means of analyzing the
algorithm of Rivas and Eddy \cite{RE:98,RE:00,Eddy:04}.
In the course of our analysis we shall make two key observations:
first, DP algorithms inevitably produce arbitrarily high crossing numbers, 
see Tab.\ref{T:taba} and second that not all $3$-noncrossing RNA structures 
can be generated by dynamic programming algorithms--at least not with the 
implemented truncations.
%%%%%%%%%%%%%%%%%%%%%%%%%%%%%%%
\begin{table}
\begin{center}
\begin{tabular}{|c|c|c|c|c|}
\hline
       $k$ & $2$ & $3$  & $4$ &$5$ \\ 
\hline growth rate &2.6180& 4.7913 & 6.8541 & 8.8875 \\ 
\hline
$k$ & $6$ &$7$ & $8$ & $9$    \\
\hline growth rate & 10.9083 & 12.9226 & 14.9330 & 16.9410 \\
\hline 
\end{tabular}
\end{center}
\centerline{}  \caption{\small The exponential growth rates of
$k$-noncrossing RNA structures (minimum arc-length greater or equal than
two).}
\label{T:taba}
\end{table}
%%%%%%%%%%%%%%%%%%%%%%%%%%%%%%%%%%%%%%
The generation of high crossing numbers is insofar problematic as it
implies a very large output class. Already for $k=4$, i.e.~for RNA structures
exhibiting three mutually crossing arcs, we have an exponential growth rate
of $6.8541$--a growth rate exceeding that of the number of natural sequences.
In other words, only for an exponentially small fraction of these structures
we will find a sequence folding into it. Remarkably, this growth rate appears 
to grow linearly in $k$, see Tab.\ref{T:taba}. Any type of study, along the 
lines of \cite{Schuster:94,Stadler:rug,Fontana:99,Reidys:2002,Stadler:98,Schuster:96,Reidys:96}, which is based on such an algorithm, is purely 
computational and does not allow to deduce generic properties in the
sense of \cite{algo-independent}.

Let us define now the non gap-matrices ($vx$, $wx$) and the
gap-matrices ($whx$, $vhx$, $zhx$ and $yhx$). 
\cite{RE:98,Reidys:frame}
The non gap-matrices, $vx$ and $wx$  are two triangular
$n\times n$ matrices, where $vx(i,j)$ is the score of the
best folding between position $i$ and $j$, provided that $i,j$ are paired
to each other and whereas $wx(i,j)$ is the score of the best folding
between the position $i$ and $j$, regardless of whether $i,j$ are
paired or not. See Tab.\ref{T:matrix}. 
%%
%%%%%%%%%%%%%%%%%%%%%%%%%%%%%%%%%%%%%%%%%%%%%%%%%%%%%%%%%%%%%%%%%%%%%%%%
\begin{figure}[ht]
\centerline{\epsfig{file=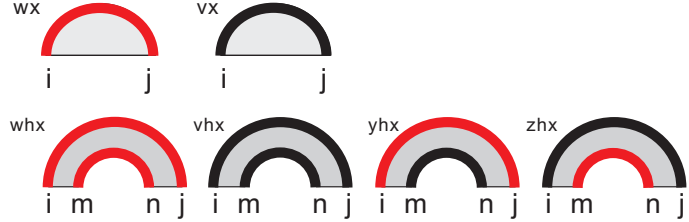,width=0.6\textwidth} \hskip8pt}
\caption{\small Non gap- and gap-matrices. The non gap-matrices $w x$, $vx$ and
gap-matrices $whx$, $vhx$, $yhx$ and $zhx$.}\label{F:matrix}
\end{figure}
%%%
%%%%%%%%%%%%%%%%%%%%%%%%%%%%%%%%%%%%%%%%%%%%%%%%%%%%%%%%%%%%%%%%%%%%%%%%
%%%
The gap-matrices are pairs of matrices, $\alpha hx(i,j;r,s)$,
where $\alpha=w,v,z,y$, are the scores of the best folding depending on the
relation between the positions $i,j$ {\it and} the relation between
positions $r,s$, respectively, see Fig.{\ref{F:matrix}}.
%%%%%%%%%%%%%%%%%%%%%%%%%%%%%%%%%%%%%%
\begin{table}
\begin{center}
\begin{tabular}{|c|c|c|c|c|c|}
\hline
Matrices & $(i,j)$ & $(r,s)$  & Matrices & $(i,j)$ & $(r,s)$ \\
\hline
$whx(i,j;r,s)$& unknown & unknown & $vhx(i,j;r,s)$ & paired & paired\\
\hline
$yhx(i,j;r,s)$ & unknown& paired & $zhx(i,j;r,s)$ & paired & unknown  \\
\hline
\end{tabular} 
\end{center} 
\centerline{}  \caption{\small Table shows the gap-matrix $whx$, 
$vhx$, $yhx$ and $zhx$.}
\label{T:matrix}
\end{table}
%%
%%%%%%%%%%%%%%%%%%%%%%%%%%%%%%%%%%%%%%%%%%%%%%%%%%%%%%%%%%%%%%%%%%%%%%%%
%%%
%%%%%%%%%%%%%%%%%%%%%%%%%%%%%%%%%%%%%%%%%%%%%%%%%%%%%%%%%%%%%%%%%%%%%%%%
%%%
\begin{figure}[ht]
\centerline{\epsfig{file=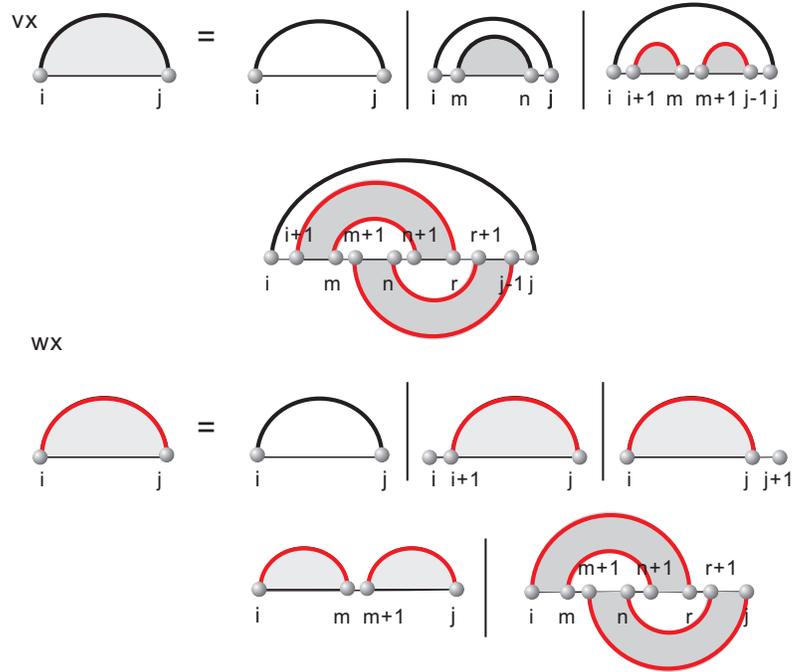,width=0.7\textwidth} \hskip8pt}
\caption{\small The basic recursions: recursion for $vx$ and $wx$ truncated at
$O(whx+whx+whx)$ in Rivas and Eddy's algorithm.}\label{F:recur}
\end{figure}
%%%
%%%%%%%%%%%%%%%%%%%%%%%%%%%%%%%%%%%%%%%%%%%%%%%%%%%%%%%%%%%%%%%%%%%%%%%%
%%%
The key idea in Rivas and Eddy's algorithm is to use gap-matrices as a
generalization of the non gap-matrices $wx$ and $vx$. In particular, both
concepts merge for $r=s-1$, where we have for any
$i\leq r\leq j$
\begin{eqnarray}
whx(i,j;r,r+1)&=&wx(i,j)\\
zhx(i,j;r,r+1)&=&vx(i,j).
\end{eqnarray}
In Fig.\ref{F:recur} we illustrate the recursion for $wx$ and $vx$ in
the pseudoknot algorithm truncated at $O(whx+whx+whx)$.
We can draw the following two conclusions:\\
$\bullet$ {\it by design}--the inductive formation of gap-matrices
          generates arbitrarily high numbers of mutually crossing
          arcs, see Fig.{\ref{F:control}.\\
$\bullet$ nonplanar, $3$-noncrossing pseudoknots cannot be generated by
          inductively forming pairs of gap-matrices, see
          Fig.{\ref{F:nonplanar}.\\
          In order to avoid any confusion: gap-matrices can and will generate
          nonplanar arc configurations, however, they can only facilitate this
          via increasing the crossing number, Fig.\ref{F:control}.
          Fig.{\ref{F:nonplanar} makes evident that the situation is more
          complex:
          nonplanarity is not tied to crossings--there are planar as well
          as nonplanar $3$-noncrossing structures.

%%
%%%%%%%%%%%%%%%%%%%%%%%%%%%%%%%%%%%%%%%%%%%%%%%%%%%%%%%%%%%%%%%%%%%%%%%%
\begin{figure}[ht]
\centerline{\epsfig{file=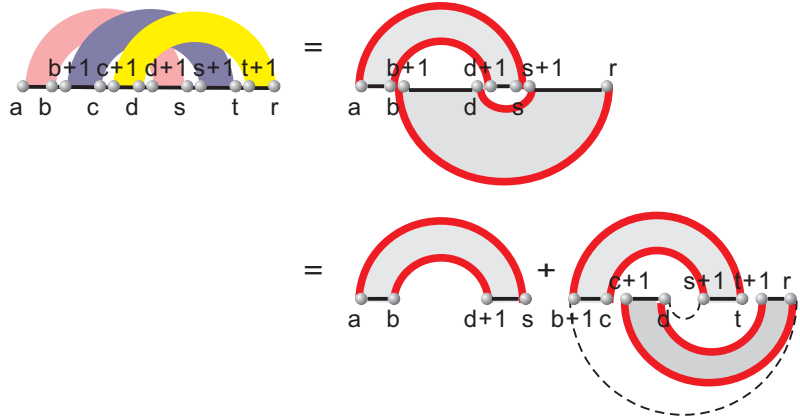,width=0.7\textwidth} \hskip8pt}
\caption{\small No control over crossings:
Here we show how to build a $4$-noncrossing RNA pseudoknot
with gap-matrices. Iterating the formation of gap-matrices will produce
higher and higher crossings.}
\label{F:control}
\end{figure}
%%%
%%%%%%%%%%%%%%%%%%%%%%%%%%%%%%%%%%%%%%%%%%%%%%%%%%%%%%%%%%%%%%%%%%%%%%%%
%%%
%%%
%%%%%%%%%%%%%%%%%%%%%%%%%%%%%%%%%%%%%%%%%%%%%%%%%%%%%%%%%%%%%%%%%%%%%%%%%%%%%%
%%%
%%%
%%%%%%%%%%%%%%%%%%%%%%%%%%%%%%%%%%%%%%%%%%%%%%%%%%%%%%%%%%%%%%%%%%%%%%%%
\begin{figure}[ht]
\centerline{\epsfig{file=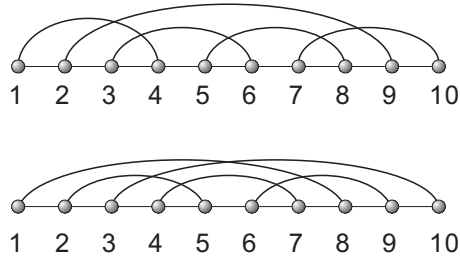,width=0.4\textwidth} \hskip8pt}
\caption{\small Two nonplanar,
$3$-noncrossing RNA structures, which cannot be
generated by pairs of gap-matrices.}\label{F:nonplanar}
\end{figure}
%%%
%%%%%%%%%%%%%%%%%%%%%%%%%%%%%%%%%%%%%%%%%%%%%%%%%%%%%%%%%%%%%%%%%%%%%%%%
%%%

\section{Specifying an output: $k$-noncrossing, canonical RNA structures}
\label{S:out}
%%%
%%%%%%%%%%%%%%%%%%%%%%%%%%%%%%%%%%%%%%%%%%%%%%%%%%%%%%%%%%%%%%%%%%%%%%%%%%%%%%
%%%

The previous section showed that, for RNA pseudoknot structures,
DP-algorithms fold into an uncontrollably large set of structures. 
This phenomenon is in vast contrast to the situation
for RNA secondary structures. The standard DP-routine cannot
produce any crossings, whence they {\it a priori} produce 
secondary structures. We now follow in the footsteps
of Waterman by generalizing his strategy for the case of secondary
structures to pseudoknot structures.
Accordingly, the first step is to specify a combinatorial output class.
To this end we shall provide some basic facts on a particular
representation of RNA structures.

A $k$-noncrossing diagram is a labeled graph over the vertex
set $[n]$ with vertex degrees $\leq 1$, represented by drawing
its vertices $1,\ldots, n$ in a horizontal line and its arcs
$(i,j)$, where $i<j$, in the upper half-plane, containing at
most $k-1$ mutually crossing arcs. The vertices and arcs
correspond to nucleotides and Watson-Crick (\textbf{A-U},
\textbf{G-C}) and (\textbf{U-G}) base pairs, respectively.
Diagrams have the following three key parameters: the maximum
number of mutually crossing arcs, $k-1$, the minimum arc-length,
$\lambda$ and minimum stack-length, $\sigma$
($(k,\lambda, \sigma)$-diagrams). The length of an arc $(i,j)$
is given by $j-i$ and a stack of length $\sigma$ is the sequence of
``parallel`` arcs of the form
\begin{equation}\label{E:stack}
((i,j),(i+1,j-1),\ldots, (i+(\sigma-1),j-(\sigma-1))),
\end{equation}
see Fig.\ref{F:ma_reid2.eps}.
%%%
%%%%%%%%%%%%%%%%%%%%%%%%%%%%%%%%%%%%%%%%%%%%%%%%%%%%%%%%%%%%%%%%%%%%
%%%
\begin{figure}[ht]
\centerline{%
\epsfig{file=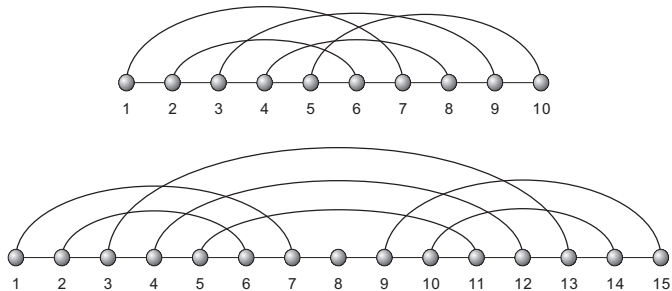,width=0.6\textwidth}\hskip15pt
 }
\caption{\small $k$-noncrossing diagrams: we display a
$4$-noncrossing, arc-length $\lambda\ge 4$ and $\sigma\ge 1$
(upper) and $3$-noncrossing, $\lambda\ge 4$ and $\sigma\ge 2$
(lower) diagram.} \label{F:ma_reid2.eps}
\end{figure}
%%%
%%%%%%%%%%%%%%%%%%%%%%%%%%%%%%%%%%%%%%%%%%%%%%%%%%%%%%%%%%%%%%%%%%%%
%%%
We call an arc of length $\lambda$ a $\lambda$-arc.

We are now in position to specify the output-set.
We shall consider RNA pseudoknot
 structures that are $3$-noncrossing, $\sigma \ge 3$-canonical
 and have a minimum arc-length $\lambda \ge 4$. The
 $3$-noncrossing property is mostly for algorithmic convenience
 and the generalization to higher crossing numbers represents
 not a major obstacle. We consider $3$-canonical structures,
 i.e.~those in which each stack has length at least three,
 since we are interested in minimum free energy structures.
 Finally, the minimum arc-length of four is a result of
 biophysical constraints. Accordingly, we shall identify
 pseudoknot RNA structures with $\langle k,4,\sigma\rangle$-diagrams and
refer to them simply as $\langle k,\sigma\rangle$-structures,
implicitly assuming the minimum arc-length $\lambda \ge 4$. In
Fig.\ref{F:phdv.eps}
%%%
%%%%%%%%%%%%%%%%%%%%%%%%%%%%%%%%%%%%%%%%%%%%%%%%%%%%%%%%%%%%%%%%%%%%
%%%
\begin{figure}[ht]
\centerline{%
\epsfig{file=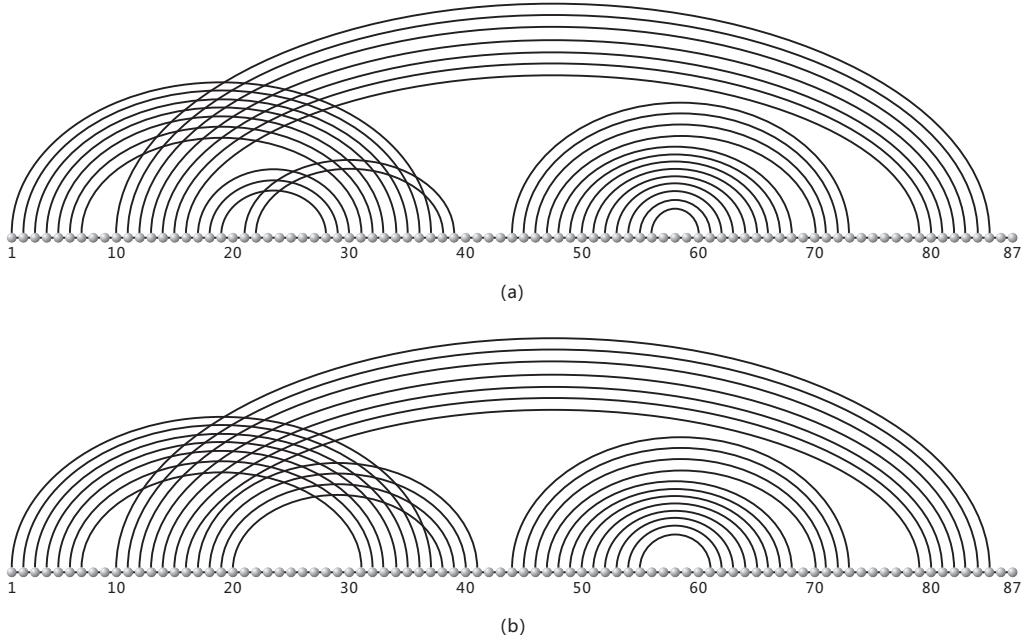,width=0.9\textwidth}\hskip15pt
 }
\caption{\small The HDV-virus pseudoknot structures as folded by
{\sf cross} (b). This structure differs from the natural structure 
displyed in (a) \cite{W:HDV} by exactly seven base pairs.} 
\label{F:phdv.eps}
\end{figure}
%%%
%%%%%%%%%%%%%%%%%%%%%%%%%%%%%%%%%%%%%%%%%%%%%%%%%%%%%%%%%%%%%%%%%%%%
%%%
we present a particular $3$-noncrossing, $3$-canonical RNA structure:
the HDV-virus as folded by {\sf cross}.

We next present some of the combinatorics of $\langle 3,\sigma
\rangle$-structures.
Let $\mathsf{T}_{k,\sigma}^{[4]}$ denote the number of
$k$-noncrossing, $\sigma$-canonical RNA structures over $[n]$.
 The generating function,
$$
\mathbf{T}_{k,\sigma}^{[4]}(z)=\sum_{n\ge 0}
 \mathsf{T}_{k,\sigma}^{[4]}(n)z^n \quad k,\sigma \ge 3
$$
of $k$-noncrossing, $\sigma$-canonical RNA structures has been
obtained in \cite{Reidys:08ma}. This function is closely related to
$\mathbf{F}_k(z)=\sum_n f_k(2n,0)z^{2n}$, the ordinary
 generating function of $k$-noncrossing matchings. Beyond
 functional equations implied directly by the
 reflection-principle \cite{Gessel:}, the following asymptotic formula
 has been derived \cite{Reidys:08wang}
\begin{equation}
\forall k\in \mathbb{N}, \quad f_k(2n,0)\sim c_k n^{
-((k-1)^2+(k-1)/2)}(2(k-1))^{2n}, \quad c_k>0.
\end{equation}\label{asy}
Setting
$$
w_0(x)=\frac{x^{2\sigma-2}}{1-x^2+x^{2\sigma}} \quad
\text{\rm and} \quad v_0(x)=1-x+w_0(x)x^2+w_0(x)x^3+w_0(x)x^4
$$
we can now state
%%%
%%%%%%%%%%%%%%%%%%%%%%%%%%%%%%%%%%%%%%%%%%%%%%%%%%%%%%%%%%%%%%%%%%%%%%%%%%
%%%
\begin{theorem}
Let $k,\sigma\in \mathbb{N}$, where $k,\sigma \ge 3$, $x$ is an
 indeterminate and $\rho_k$ the dominant, positive real
singularity of $\mathbf{F}_k(z)$. Then $\mathbf{T}_{k,\sigma}^{[4]}(x)$,
the generating function of $\langle k,\sigma\rangle$-structures, is given by
\begin{equation}
 \mathbf{T}_{k,\sigma}^{[4]}(x)=\frac{1}{v_0(x)}\mathbf{F}_k
\left( \frac{\sqrt{w_0(x)}x}{v_0(x)} \right).
\end{equation}
Furthermore, the asymptotic formula
\begin{equation}
 \mathbf{T}_{k,\sigma}^{[4]}(n) \sim c_k n^{-(k-1)^2-(k-1)/2}
\left( \frac{1}{\gamma_{k,\sigma}^{[4]}} \right)^n,\ \quad \
\text{\rm for} \quad k=3,4,\ldots,9.
\end{equation}
holds, where $\gamma_{k,\sigma}^{[4]}$  is the minimal positive
real solution of the equation
$\frac{\sqrt{w_0(x)}x}{v_0(x)}=\rho_k$.
\end{theorem}
%%%
%%%%%%%%%%%%%%%%%%%%%%%%%%%%%%%%%%%%%%%%%%%%%%%%%%%%%%%%%%%%%%%%%%%%%%%%%%
%%%
Theorem $1$ implies exact enumeration results as well as an array of
exponential growth rates indexed by $k$ and $\sigma$. The latter are
presented in Tab.\ref{T:1} and are of relevance in the
context of the asymptotic analysis of the algorithm.
%%
%%%%%%%%%%%%%%%%%%%%%%%%%%%%%%%%%%%%%%%%%%%%%%%%%%%%%%%%%%%%%%%%%%%%%%%%
\begin{table}
\begin{center}
\begin{tabular}{|c|c|c|c|c|c|c|c|}
%\hline
%  \multicolumn{9}{|c|}{\textbf{$\lambda=2$}}\\
\hline $k$ & \small$3$ & \small $4$ & \small $5$ & \small $6$
&\small $7$ & \small $8$ & \small $9$  \\
\hline $\sigma=3$ & \small$2.0348$ & \small $2.2644$ & \small
$2.4432$ & \small $2.5932$
&\small $2.7243$ & \small $2.8414$ & \small $2.9480$  \\
$\sigma=4$ & \small$1.7898$ & \small $1.9370$ & \small $2.0488$ &
\small $2.1407$
&\small $2.2198 $ & \small $2.2896$ & \small $2.3523$  \\
$\sigma=5$ & \small$1.6465$ & \small $1.7532$ & \small $1.8330$ &
\small $1.8979$
&\small $1.9532$ & \small $2.0016$ & \small $2.0449$  \\
$\sigma=6$ & \small$1.5515$ & \small $1.6345$ & \small $1.6960$ &
\small $1.7457$
&\small $1.7877$ & \small $1.8243$ & \small $1.8569$  \\
 $\sigma=7$ & \small$1.4834$ & \small $1.5510$ & \small
$1.6008$ & \small $1.6408$
&\small $1.6745$ & \small $1.7038$ & \small $1.7297$  \\
$\sigma=8$ & \small$1.4319$ & \small $1.4888$ & \small $1.5305$ &
\small $1.5639$
&\small $1.5919$ & \small $1.6162$ & \small $1.6376$  \\
 $\sigma=9$ & \small$1.3915$ & \small $1.4405$ & \small
$1.4763$ & \small $1.5049$
&\small $1.5288$ & \small $1.5494$ & \small $1.5677$  \\
\hline
\end{tabular}
\centerline{}   \caption{\small Exponential growth rates of
$\langle k,\sigma\rangle$-structures.
}\label{T:1}
\end{center}
\end{table}
%%
%%%%%%%%%%%%%%%%%%%%%%%%%%%%%%%%%%%%%%%%%%%%%%%%%%%%%%%%%%%%%%%%%%%%%%%%
%%%%%%%%%%%%%%%%%%%%%%%%%%%%%%%
In addition, Tab.\ref{T:1} shows that $3$-noncrossing, $\sigma$-canonical RNA
structures have remarkably moderate growth rates. $\sigma$-canonical structures
with higher crossing numbers exhibit also moderate growth rates, indicating
that generalizations of the current implementation of {\sf cross} from $k=3$
to $k=4$ or $5$ are feasible.

%%%
%%%%%%%%%%%%%%%%%%%%%%%%%%%%%%%%%%%%%%%%%%%%%%%%%%%%%%%%%%%%%%%%%%%%%%%%%%
%%%

\section{Loops, motifs and shadows}

%%%
%%%%%%%%%%%%%%%%%%%%%%%%%%%%%%%%%%%%%%%%%%%%%%%%%%%%%%%%%%%%%%%%%%%%%%%%%%
%%%

Suppose we are given a $\langle 3,\sigma\rangle$-structure, $S$. Let 
$\alpha$ be an $S$-arc and denote the set of $S$-arcs that cross $\beta$
by $\mathscr{A}_{S}(\beta)$. Clearly we have
\begin{equation}
\beta \in \mathscr{A}_S(\alpha) \quad
 \Longleftrightarrow \quad  \alpha\in
 \mathscr{A}_S(\beta).
\end{equation}
An arc $\alpha \in \mathscr{A}_S(\beta)$ is called a
minimal, $\beta$-crossing if there exists no $\alpha' \in
\mathscr{A}_S(\beta)$ such that $\alpha' \prec \alpha$. Note that
$\alpha \in \mathscr{A}_S(\beta)$ can be minimal $\beta$-crossing,
while $\beta$ is {\it not} minimal $\alpha$-crossing.
We call a pair of crossing arcs $(\alpha,\beta)$ balanced, if $\alpha$ is
minimal, $\beta$-crossing and $\beta$ is minimal $\alpha$-crossing,
respectively. $3$-noncrossing diagrams exhibit 
the following four basic loop-types $3$-noncrossing diagrams:

\begin{figure}[ht]
\centerline{%
\epsfig{file=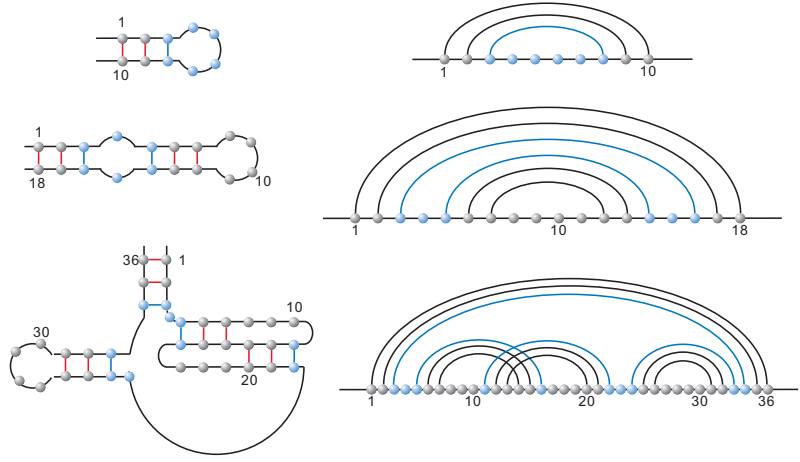, width=0.7\textwidth}\hskip15pt
 }
\caption{\small The standard loop-types: hairpin-loop (top),
interior-loop (middle) and multi-loop (bottom).
These represent all loop-types that occur in RNA secondary
structures.} \label{F:stand}
\end{figure}
%%%
%%%%%%%%%%%%%%%%%%%%%%%%%%%%%%%%%%%%%%%%%%%%%%%%%%%%%%%%%%%%%%%%%%%%%%%%%%
%%%
{\bf (1)} a {\it hairpin}-loop, being a pair
$$
((i,j),[i+1,j-1])
$$
where $(i,j)$ is an arc and $[i,j]$ is an interval, i.e.~a sequence of
consecutive vertices $(i,i+1,\dots,j-1,j)$. \\
{\bf (2)} an {\it interior}-loop, being a sequence
$$
((i_1,j_1),[i_1+1,i_2-1],(i_2,j_2),[j_2+1,j_1-1]),
$$
where $(i_2,j_2)$ is nested in $(i_1,j_1)$.\\
{\bf (3)} a {\it multi}-loop, see Fig.\ref{F:stand}, being a sequence
$$
((i_1,j_1),[i_1+1,\omega_1-1],S_{\omega_{1}}^{\tau_1},
[\tau_1+1,\omega_2-1], S_{\omega_{2}}^{\tau_2},
\dots )
$$
where $S_{\omega_h}^{\tau_h}$ denotes a pseudoknot structure
over $[\omega_h,\tau_h]$ (i.e.~nested in $(i_1,j_1)$) and
subject to the following condition: if all
$S_{\omega_h}^{\tau_h}=(\omega_h,\tau_h)$, i.e.~all substructures
are simply arcs, for all $h$, then $h\ge 2$.\\
%%%
%%%%%%%%%%%%%%%%%%%%%%%%%%%%%%%%%%%%%%%%%%%%%%%%%%%%%%%%%%%%%%%%%%%%%%%%%%
%%%
\begin{figure}[ht]
\centerline{\epsfig{file=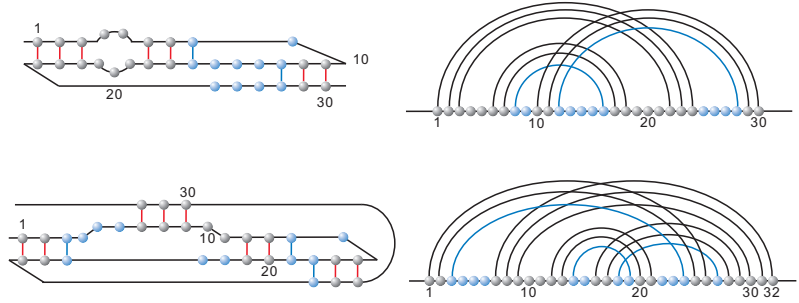,width=0.7\textwidth}\hskip8pt}
\caption{\small Pseudoknots: we display a balanced (top)
and an unbalanced pseudoknot (bottom). The latter contains the stack over
$(3,24)$, which is minimal for the arc $(9,30)$, which is {\it not} 
contained in the pseudoknot.}
\label{F:pseudoloop}
\end{figure}
%%%
%%%%%%%%%%%%%%%%%%%%%%%%%%%%%%%%%%%%%%%%%%%%%%%%%%%%%%%%%%%%%%%%%%%%%%%%%%
%%%
We finally define pseudokont-loops:\\
{\bf (4)} a {\it pseudoknot}, see Fig.\ref{F:pseudoloop},
consists of the following data:\\
({\sf P1}) a set of arcs
$$
P=\left\{(i_1,j_1),(i_2,j_2), \dots,(i_t,j_t)\right\},
$$
where $i_1=\min\{i_s\}$
and $j_t=\max\{j_s\}$, such that \\
{\bf (i)} the diagram induced by the arc-set $P$ is irreducible,
i.e.~the line-graph of $P$ is connected and \\
{\bf (ii)} for each $(i_{s},j_{s})\in P$ there exists some arc $\beta$
(not necessarily contained in $P$) such that $(i_{s},j_{s})$
is minimal $\beta$-crossing.\\
({\sf P2})  all vertices $i_1<r<j_t$, not contained in hairpin-, interior-
or {multi-loops}.\\
We call a pseudoknot balanced if its arc-set can be decomposed into
pairs of balanced arcs.
%%%
%%%%%%%%%%%%%%%%%%%%%%%%%%%%%%%%%%%%%%%%%%%%%%%%%%%%%%%%%%%%%%%%%%%%%%%%%%
%%%

\subsection{Motifs and shadows}

%%%
%%%%%%%%%%%%%%%%%%%%%%%%%%%%%%%%%%%%%%%%%%%%%%%%%%%%%%%%%%%%%%%%%%%%%%%%%
%%%

Let $\prec$ denote the partial order over the set of arcs
(written as $(i,j),\ i<j)$ of a $k$-noncrossing diagram,
given by
\begin{equation}
 (i_1,j_1) \prec (i_2,j_2) \ \Longleftrightarrow \
 i_2<i_1 \ \wedge \ j_1<j_2.
\end{equation}
A $k$-noncrossing core is a $k$-noncrossing diagram without any
two arcs of the form $(i,j),\ (i+1,j-1)$. Any $k$-noncrossing
RNA structure, $S$ has a unique $k$-noncrossing core, $c(S)$ 
\cite{07lego}, obtained in two steps:
first one identifies all arcs contained in stacks, inducing a
contracted diagram and secondly one relabels the vertices.
%%%
%%%%%%%%%%%%%%%%%%%%%%%%%%%%%%%%%%%%%%%%%%%%%%%%%%%%%%%%%%%%%%%%%%%
%%%
\begin{figure}[ht]
\centerline{\epsfig{file=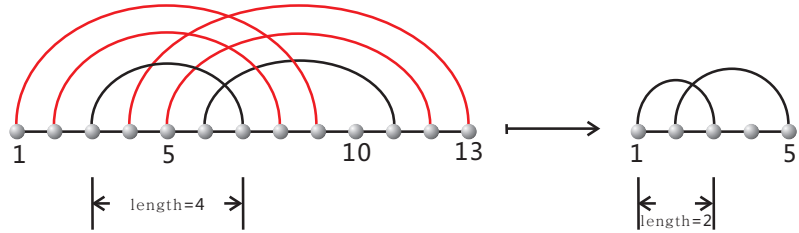,width=0.7\textwidth}\hskip15pt}
\caption{\small Core-structures: A structure, $S$, (lhs) is mapped
into its core $c(S)$ (rhs). Clearly $S$ has arc-length $\ge 4$
and as a consequence of the collapse of the stack
$((4,13),(5,12),(6,11))$ (the red arcs are being removed)
into the arc $(2,5)$. $c(S)$ contains the arc $(1,3)$. This arc
becomes, after relabeling, a $2$-arc.}\label{F:ma_reid5.eps}
\end{figure}
%%%
%%%%%%%%%%%%%%%%%%%%%%%%%%%%%%%%%%%%%%%%%%%%%%%%%%%%%%%%%%%%%%%%%%%
%%%
Note that the core-map does in general not preserve arc-length.
%%%
%%%%%%%%%%%%%%%%%%%%%%%%%%%%%%%%%%%%%%%%%%%%%%%%%%%%%%%%%%%%%%%%%%%
%%%
\begin{definition}
\textbf{(Motif)} A $\langle k,\sigma\rangle$-motif, $\mathfrak{m}$, is a
$\langle k,\sigma\rangle$-structure over $[n]$, having the following
properties: \\
{\sf (M1)} $\mathfrak{m}$ has a nonnesting core. \\
{\sf (M2)} All $\mathfrak{m}$-arcs are contained in stacks of
 length exactly $\sigma\ge 3$ and length $\lambda \ge 4$.\\
The set of all motifs is denoted by $\mathbb{M}^\sigma_k(n)$
and we set $\mu^*_{k,\sigma}(n) = |\mathbb{M}^\sigma_k(n)|$.
\end{definition}
%%%
%%%%%%%%%%%%%%%%%%%%%%%%%%%%%%%%%%%%%%%%%%%%%%%%%%%%%%%%%%%%%%%%%%%
%%%
Property {\sf (M1)} is obviously equivalent to: all arcs of
the core, $c(\mathfrak{m})$, are $\prec$-maximal.

Let $S$ be a $\langle 3,\sigma\rangle$-structure.
We call two $k$-noncrossing diagrams $\delta_1,\delta_2$ adjacent if
and only if $\delta_2$ is derived by selecting a pair of
isolated $\delta_1$-vertices, $i<j$ such that $(i-1,j+1)$
is a $\delta_1$-arc. With respect to this notion of adjacency
the set of $k$-noncrossing diagrams over $[n]$ becomes a
directed graph, which we denote by $\mathscr{G}_k(n)$.
%%%
%%%%%%%%%%%%%%%%%%%%%%%%%%%%%%%%%%%%%%%%%%%%%%%%%%%%%%%%%%%%%%%%%%%%%%%%%
%%%
\begin{definition}{\bf (Shadow)}
A shadow of $S$ is a $\mathscr{G}_k(n)$-vertex connected
to $S$ by a $\mathscr{G}_k(n)$-path.
\end{definition}
%%%
%%%%%%%%%%%%%%%%%%%%%%%%%%%%%%%%%%%%%%%%%%%%%%%%%%%%%%%%%%%%%%%%%%%%%%%%%
%%%
Intuitively speaking, a shadow is derived by extending the stacks of a 
structure from top to bottom.
%%%
%%%%%%%%%%%%%%%%%%%%%%%%%%%%%%%%%%%%%%%%%%%%%%%%%%%%%%%%%%%%%%%%%%%%%%%%%
%%%
\begin{theorem}\label{T:decompose}
Suppose $k,\sigma\ge 2$. \\
{\rm (a)}  Any $k$-noncrossing, $\sigma$-canonical RNA
 structure corresponds to a unique sequence of shadows.\\
{\rm (b)}  Any $\langle3,\sigma\rangle$-structure has a unique
 loop-decomposition.
\end{theorem}
%%%
%%%%%%%%%%%%%%%%%%%%%%%%%%%%%%%%%%%%%%%%%%%%%%%%%%%%%%%%%%%%%%%%%%%%%%%%%
%%%

\begin{proof}
Ad {\rm (a)}. Suppose $S$ is an arbitrary
$\langle k,\sigma\rangle$-structure over $[n]$. We
prove the theorem by induction on the number of
$S$-arcs. We consider the set of
$\prec$-maximal elements,
$
S^*=\{(i,j)\mid (i,j) \ \text{\rm is $\prec$-maximal}\}.
$
Clearly, $S^*$ induces a unique $\langle k,\sigma\rangle$-motif,
$\mathfrak{m}_{k,\sigma}(S)$, contained
in $S$. Indeed, since $S$ is by
 assumption $\sigma$-canonical, each $S^*$-arc occurs in a
 stack of size $\ge \sigma$. By definition, any
 $S$-arc which is contained in a stack containing
 an (unique) $S^*$-arc is an arc of an unique shadow,
 $\overline{\mathfrak{m}}_{k,\sigma}(S)$.
 Removing all arcs contained in
 $\overline{\mathfrak{m}}_{k,\sigma}(S)$ the
 remaining diagram is still $k$-noncrossing and
 $\sigma$-canonical. To see this it suffices to observe that
 any $S$-arc not contained in
 $\overline{\mathfrak{m}}_{k,\sigma}(S)$ is
 contained in a stack of size $\ge \sigma$ not containing
 any
 $\overline{\mathfrak{m}}_{k,\sigma}(S)$-arcs.
 Assertion (a) follows now by induction on the number of
 arcs. \\
Ad {\rm (b)}. Let $c(S)$ be the core of
$S$. We shall color the
 $c(S)$-arcs, $\alpha=(i,j)$, as follows: \\
%%%%%%%%%%%%%%%%%%%%%%%%%%%%%%%%%%%%%%%%%%%%%%%%%%%%%%%%%%%%%%%%%%%%
Case (1): $\mathscr{A}_{c(S)}(\alpha)
 \neq \varnothing$. \\
%%%%%%%%%%%%%%%%%%%%%%%%%%%%%%%%%%%%%%%%%%%%%%%%%%%%%%%%%%%%%%%%%%%%%
Since $c(S)$ is a
 $3$-noncrossing diagram, we have for any two $(i,j),(i',j')
\in \mathscr{A}_{c(S)}(\beta)$, either $(i,j)\prec
(i',j')$ or $j<i'$. Therefore for any $\beta \in
 \mathscr{A}_{c(S)}(\alpha)$ there exists an
 unique $\prec$-minimal arc $\alpha^*∗ \in
\mathscr{A}_{c(S)}(\beta)$ that is nested in
$\alpha$. If there exists some $\beta$ for which
$\alpha = \alpha^*(\beta)$ holds, i.e.~$\alpha$ itself is
 minimal in $\mathscr{A}_{c(S)}(\beta)$, then we
 color $\alpha$ red. In other words, red arcs are minimal
with respect to some crossing $\beta$. Otherwise, for any $\beta \in
 \mathscr{A}_{c(S)}(\alpha)$ there exists some
$\alpha^*(\beta)\prec \alpha$. If $\alpha^*(\beta)$ is the unique 
$\prec$-maximal substructure nested in $\alpha$, then we color 
$\alpha$ green and blue, otherwise. \\
%%%%%%%%%%%%%%%%%%%%%%%%%%%%%%%%%%%%%%%%%%%%%%%%%%%%%%%%%%%%%%%%%%%%%%%%%
Case (2): $\mathscr{A}_{c(S)}(\alpha)=
\varnothing$, i.e. $\alpha$ is noncrossing in
 $c(S)$. \\
%%%%%%%%%%%%%%%%%%%%%%%%%%%%%%%%%%%%%%%%%%%%%%%%%%%%%%%%%%%%%%%%%%%%%%%%%
If there exists no
 $c(S)$-arc $\alpha' \prec \alpha$, then we color $\alpha$
 purple, if there exists exactly one maximal
 $c(S)$-arc $\alpha' \prec \alpha$, we color
$\alpha$ green and blue, otherwise.
It follows now by induction on the number of
 $c(S)$-arcs that this procedure generates a
 well defined arc-coloring. Let $i\in [n]$ be a vertex. We
 assign to $i$ either the color of the minimal non-red
 $c(S)$-arc $(r,s)$ for which $r<i<s$ holds, or
 red if there exist only red $c(S)$-arcs, $(r,s)$
 with $r<i<s$ and black, otherwise. By construction, this
 induces a vertex-arc coloring with the property of correctly
identifying all hairpin- (purple arcs and vertices), interior-
(green arcs and vertices), multi- (blue
arcs and vertices) and pseudoknot (red arcs and vertices).
\end{proof}

%%%
%%%%%%%%%%%%%%%%%%%%%%%%%%%%%%%%%%%%%%%%%%%%%%%%%%%%%%%%%%%%%%%%%%%%%%%%%%
%%%

\begin{figure}[ht]
\centerline{\epsfig{file=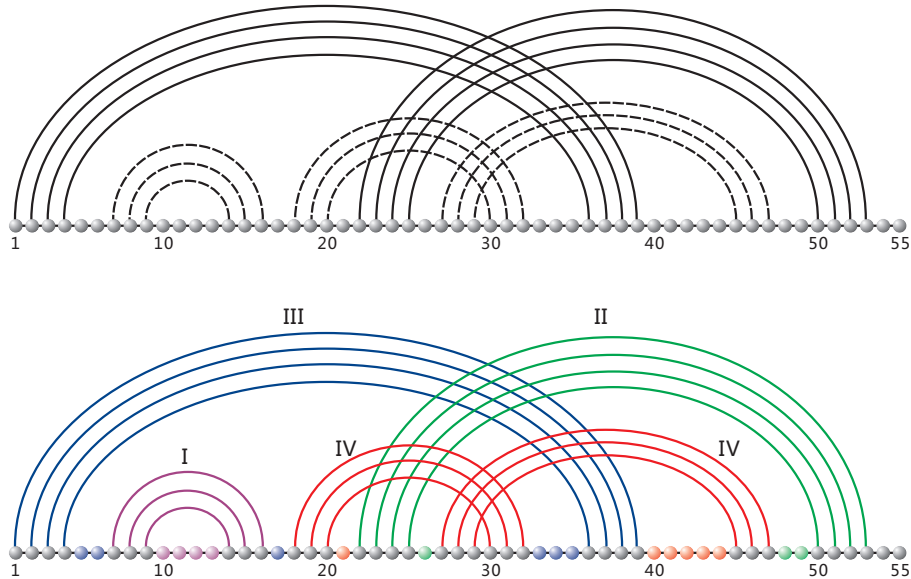,width=0.8\textwidth}\hskip8pt}
\caption{\small Shadows and loops: we give the sequence of shadows (top) 
and the loop-decomposition (below) illustrating Theorem~\ref{T:decompose}. 
Here {\sf I} (purple) is a hairpin-loop, {\sf II} (green)
represents an interior-loop, {\sf III} (blue) is a multi-loop 
and finally {\sf IV} (red) is a (balanced) pseudoknot.} 
\label{F:decompose}
\end{figure}

%%%
%%
In Fig.\ref{F:decompose} we show how these decompositions work.

%%%
%%%%%%%%%%%%%%%%%%%%%%%%%%%%%%%%%%%%%%%%%%%%%%%%%%%%%%%%%%%%%%%%%%%
%%%

\section{Phase I: motif-generation} \label{S:motifs}

%%%
%%%%%%%%%%%%%%%%%%%%%%%%%%%%%%%%%%%%%%%%%%%%%%%%%%%%%%%%%%%%%%%%%%%
%%%

The first step in {\sf cross} consists in creating some kind of shelling
of a $3$-noncrossing, canonical structure via motifs.
One key idea in {\sf cross} is the identification
of motifs as building blocks. The key point here is that, despite
the fact that motifs exhibit complicated crossings, they can be
{\it inductively} generated. This is remarkable and a result of
considering the ``dual'' of a motif which turns out to be a
restricted Motzkin-path. The latter is obtained via the bijection 
of Proposition~\ref{P:dual} between crossing and nesting arcs.

A Motzkin-path is composed by {\it up}-, {\it down}- and 
{\it horizontal}-steps. It starts at the origin, stays in the upper
halfplane and ends on the $x$-axis. Let ${\rm Mo}_k^\sigma(n)$
denote the following set of Motzkin-paths: \\
{\bf (a)} the paths have height $\le \sigma(k-1)$\\
{\bf (b)} all up- and down-steps come only in sequences of length
          $\sigma$ \\
{\bf (c)} all plateaux at height $\sigma$ have length $\ge 3$.\\
Let $\mu_{k-1,\sigma}(n)$ denote the number of
Motzkin-paths of length $n$ that (a') have height $\le \sigma(k-2)$,
(b') up- and down-steps come only in sequences of length $\sigma$.
We set for arbitrary $k, \sigma \ge 2$
\begin{eqnarray*}
G^*_{k,\sigma}(z) & = & \sum_{n\ge 0}\mu^*_{k,\sigma}(n)z^n \\
G_{k-1,\sigma}(z) & = & \sum_{n\ge 0}\mu_{k-1,\sigma}(n)z^n  \\
G_{1,\sigma}(z) & = & \frac{1}{1-z}.
\end{eqnarray*}
Now we are in position to give the main result of this section:
%%%
%%%%%%%%%%%%%%%%%%%%%%%%%%%%%%%%%%%%%%%%%%%%%%%%%%%%%%%%%%%%%%%%%%%
%%%
\begin{proposition}\label{P:dual}
Suppose $k,\sigma \ge 2$, then the following assertions hold:\\
{\rm (a)} There exists a bijection
\begin{equation}
 \beta:\mathbb{M}_k^\sigma(n)\longrightarrow
\text{\rm Mo}_k^\sigma(n) .
\end{equation}
{\rm (b)} We have the following recurrence equations
\begin{eqnarray}\label{E:mu1}
\mu^*_{k,\sigma}(n) & = & \mu^*_{k,\sigma}(n-1)+\sum_{s=0}^
{n-(2\sigma+3)} \mu_{k-1}(n-2\sigma-s)
\mu^*_{k,\sigma}(s) \quad \text{\rm for} \ n> 2\sigma \\
\label{E:mu2}
 \mu_{k,\sigma}(n) & = & \mu_{k,\sigma}(n-1)+\sum_{s=0}^
{n-2\sigma} \mu_{k-1}(n-2\sigma-s)
\mu_{k,\sigma}(s) \qquad \text{\rm for} \ n> 2\sigma-1.
\end{eqnarray}
where $\mu^*_{k,\sigma}(n)=1$ for $0\le n \le 2\sigma$ and
$\mu_{k-1,\sigma}(n)=1$ for $0\le n \le 2\sigma-1$.\\
{\rm (c)} We have the following formula for the generating functions
\begin{eqnarray}
\label{mu3}
 G^*_{k,\sigma}(z) & = & 
                    \frac{1}{1-z-z^{2\sigma}(G_{k-1,\sigma}(z)-(z^2+z+1))}\\ 
\label{mu3'}
 G_{k-1,\sigma}(z) & = & \frac{1}{1-z-z^{2\sigma}G_{k-2,\sigma}(z)}.
\end{eqnarray}
and, in particular, for $k=3$ we have the following asymptotic formula
\begin{equation}\label{mu4}
\mu^*_{3,\sigma}(n) \sim c_\sigma \left( \frac{1}
{\zeta_\sigma}\right)^n,
\end{equation}
where $c_\sigma$ and $\zeta_\sigma^{-1}$ are given by Tab.\ref{T:tabc}.
\end{proposition}
%%%%%%%%%%%%%%%%%%%%%%%%%%%%%%%
\begin{table}
\begin{center}
\begin{tabular}{|c|c|c|c|c|c|c|c|}
\hline
       $\sigma$ & $2$  & $3$ &$4$ &$5$ &$6$ & $7$    \\
\hline $\zeta^{-1}_{\sigma}$ & 1.7424 &
1.5457 & 1.4397 & 1.3721 & 1.3247 & 1.2894  \\
\hline $c_\sigma$ & 0.1077 & 0.0948 & 0.0879 &
0.0840 & 0.0804 & 0.0780 \\
\hline
\end{tabular}
\end{center}
\centerline{}  \caption{\small The exponential growth rates of
$\mu^*_{3,\sigma}(n)$} \label{T:tabc}
\end{table}
%%%%%%%%%%%%%%%%%%%%%%%%%%%%%%%%%%%%%%
%%%
%%%%%%%%%%%%%%%%%%%%%%%%%%%%%%%%%%%%%%%%%%%%%%%%%%%%%%%%%%%%%%%%%%%
%%%
\begin{proof}
Let $\mathfrak{m}$ be a $\langle k,\sigma\rangle$-motif.
We construct the bijection $\beta$ as follows: reading the
vertex labels
 of $\mathfrak{m}$ in increasing order we map each
$\sigma$-tuple of origins and termini into a $\sigma$-tuple
 of \textit{up-}steps and \textit{down-}steps,
 respectively. Furthermore isolated points are mapped into
 \textit{horizontal-}steps. The resulting paths are by
 construction Motzkin-paths of height $\le \sigma(k-1)$.
 Since motifs have arcs of length $\ge 4$ the paths have
at height $\sigma$ plateaux of length $\ge 3$. In addition
we have
 $\sigma$-tuples of up- and down-steps. Therefore $\beta$
 is well defined. To see that $\beta$ is bijective we
 construct its inverse explicitly. Consider an element
$\zeta\in \text{\rm Mo}^\sigma_k(n)$. We shall pair
 $\sigma$-tuples of up-steps and down-steps as follows:
 starting from left to  right we pair the first up-step with
 the first down-step tuple and proceed inductively, see 
Fig.\ref{F:mozi}.
%%%
%%%%%%%%%%%%%%%%%%%%%%%%%%%%%%%%%%%%%%%%%%%%%%%%%%%%%%%%%%%%%%%%%%%
%%%
\begin{figure}[ht]
\centerline{\epsfig{file=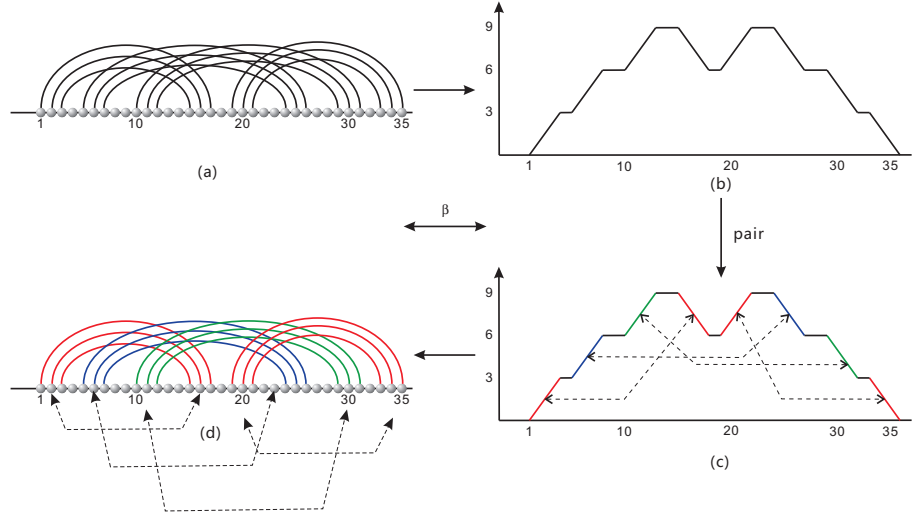,width=0.8\textwidth}\hskip8pt}
\caption{\small The bijection $\beta$: First we have a map 
from {\sf (a)} to {\sf (b)}. Then we pair the $\sigma$-tuples of up-steps 
and down-steps, see the vertical map from {\sf (b)} to {\sf (c)}. 
The so derived pairs, see the horizontal map from {\sf (c)} to {\sf (d)},
allow to reconstruct the original motif.}
\label{F:mozi}
\end{figure}
%%%
%%%%%%%%%%%%%%%%%%%%%%%%%%%%%%%%%%%%%%%%%%%%%%%%%%%%%%%%%%%%%%%%%%%
%%%
It is
 clear from the definition of Motzkin-paths that this pairing
 procedure is well defined. Each such pair
$$
((u_i,u_{i+1},\ldots,u_{i+\sigma},(d_j,d_{j+1},\ldots,
j_{j+\sigma}))
$$
corresponds uniquely to the sequence of arcs
$((i+\sigma,j),\ldots,(i,j+\sigma))$ from which we can
 conclude that $\zeta$ induces a unique $\sigma$-canonical
 diagram, $\delta_\zeta$ over $[n]$. Furthermore
 $\delta_\zeta$ has by construction a nonnesting core. A
 diagram contains a $k$-crossing if and only if it contains
 a sequence of arcs $(i_1,j_1),\ldots,(i_k,j_k)$ such that
 $i_1<i_2<\cdots<i_k<j_1<j_2<\cdots<j_k$. Therefore
 $\delta_\zeta$ is $k$-noncrossing if and only if its
 underlying path $\zeta$ has height $< \sigma k$. We
 immediately derive $\beta(\delta_\zeta) = \zeta$, whence
$\beta$ is a bijection. Using the Motzkin-path
 interpretation we immediately observe that
$\text{\rm Mo}_k^\sigma(n)$-paths can be constructed
 recursively from paths that start with a horizontal-step or
 an up-step, respectively. The recursions eq.~(\ref{E:mu1}) and
 eq.~(\ref{E:mu2}) and the
 generating functions of eq.~(\ref{mu3}) and eq.~(\ref{mu3'}) 
are straightforwardly derived.
As for the particular case $G^*_{3,\sigma}(z)$, we have
\begin{equation}
 G^*_{3,\sigma}(z)=\frac{1}{1-z-z^{2\sigma}\left[\frac{1}
{1-z-z^{2\sigma}[\frac{1}{1-z}]}-(z^2+z+1)\right]}.
\end{equation}
The unique dominant, real singularities of
 $G^*_{3,\sigma}(z)$ are simple poles, denoted by
$\zeta_\sigma$. Being a rational function,
 $G^*_{k,\sigma}(z)$ admits a partial fraction expansion
$$
G^*_{k,\sigma}(z)=H(z)+\sum_{(\zeta,r)}\frac{c_{(\zeta,r)}}
{(\zeta-z)^r}
$$
and eq.~(\ref{mu4}) follows in view of
\begin{equation}
[z^n]\frac{1}{\zeta-z}=\frac{1}{\zeta}[z^n]\frac{1}
{1-z/\zeta}=\frac{1}{\zeta}{n \choose 0}\left(\frac{1}
{\zeta}\right)^n=\left(\frac{1}{\zeta}\right)^{n+1}.
\end{equation}
\end{proof}

%%%
%%%%%%%%%%%%%%%%%%%%%%%%%%%%%%%%%%%%%%%%%%%%%%%%%%%%%%%%%%%%%%%%%%%%%%%%%%%%%%
%%%

\section{Phase II: the skeleta-tree} \label{S:skeleta}

%%%
%%%%%%%%%%%%%%%%%%%%%%%%%%%%%%%%%%%%%%%%%%%%%%%%%%%%%%%%%%%%%%%%%%%%%%%%%%%%%%
%%%

In this section we enter the second phase of {\sf cross}. What will happen
here, is that each irreducible shadow, generated during the first phase
described in Section~\ref{S:motifs}, gives rise to a tree of skeleta.
The intuition behind this construction is that each tree-vertex, i.e.~each
skeleton, represents a maximal ``non-inductive'' arc configuration.
This does not mean that a skeleton contains all crossings arcs of the final
structure, but all further crossings are derived by adding independent
substructures. In other words: their energy contributions are additive.

A skeleton, $S$, is a $3$-noncrossing structure whose core has no noncrossing
arcs, i.e.~for any arc $\alpha$ we have 
$\mathscr{A}_S(\alpha)\neq \varnothing$,
see Fig.\ref{F:IS}. In addition, in a skeleton over the segment 
$\{i,i+1,\dots,j-1,j\}$, $S_{i,j}$, the positions $i$ and $j$ are paired.
Recall that an interval is a sequence of consecutive, unpaired bases 
$(i,i+1,\cdots ,j)$, where $i-1$ and $j+1$ are paired.
Furthermore, recall that a stack of length $\sigma$ (see eq.~(\ref{E:stack})) 
is a sequence of parallel arcs
$
((i,j),(i+1,j-1),\ldots, (i+(\sigma-1),j-(\sigma-1)))
$,
which we write as $(i,j,\sigma)$. Note that
$\sigma\geq \sigma_0$, where $\sigma_0$ is the minimum stack length
of the structure, see Fig.\ref{F:IS}. An irreducible shadow over 
$\{i,i+1,\dots,j-1,j\}$ is denoted by $IS_{i,j}$. It is a particular 
skeleton, i.e.~a skeleton in which there are no nested arcs.

{\begin{remark}
In our implementation of {\sf cross}, the number of stacks of an irreducible
shadow is an input parameter. As default we set its maximum value to be three.
\end{remark}
%%%
%%%%%%%%%%%%%%%%%%%%%%%%%%%%%%%%%%%%%%%%%%%%%%%%%%%%%%%%%%%%%%%%%%%%%%%%%%%%%%
%%%
\begin{figure}[ht]
\centerline{\epsfig{file=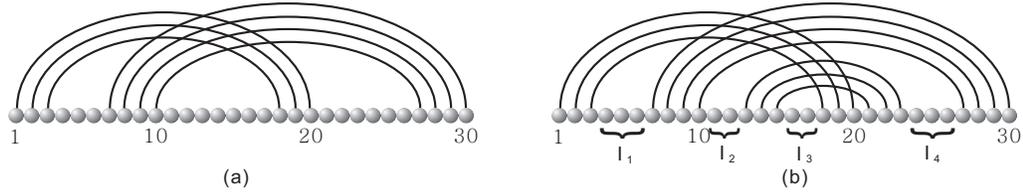,width=0.9\textwidth}\hskip8pt}
\caption{\small Irreducible shadows and skeleta: an irreducible shadow (a), 
containing the stack $(1,20,3)$ and $(7,30,4)$. (b) A skeleton drawn with 
its four induced intervals $I_1,I_2,I_3,I_4$.} \label{F:IS}
\end{figure}
%%%
%%%%%%%%%%%%%%%%%%%%%%%%%%%%%%%%%%%%%%%%%%%%%%%%%%%%%%%%%%%%%%%%%%%%%%%%%%%%%%
%%%
We are now in position to construct the skeleta-tree.
Suppose we are given a $3$-noncrossing skeleton, $S$. We label the
$S$-intervals $\{I_1,\dots,I_m\}$ from left to right and consider pairs
$(S,r)$, where $r$ is an integer $1\le r\le m-1$.
Given a pair $(S,r)$ we construct new pairs $(S',r')$ where $r'\geq r$ as
follows:
we replace a pair of intervals $(I_p,I_q)$, $i\in I_p,j\in I_q$, $i\ge r$
by the stack $\alpha=(i,j,\sigma)$, subject to the following
conditions
\begin{itemize}
\item $S'$ is a $3$-noncrossing skeleton
\item $(i+\sigma-1,j-\sigma+1)$ is a minimal element in $(S',\prec)$
\item $r'$ is the label of the first paired base preceding the
      interval $I_p$.
\item $i-1$ and $j+1$ are not paired to each other.
\end{itemize}
Fig.\ref{F:insert} displays the two basic scenarios via which stacks
are being inserted.
%%%%
%%%%%%%%%%%%%%%%%%%%%%%%%%%%%%%%%%%%%%%%%%%%%%%%%%%%%%%%%%%%%%%%%%%%%%%%%
%%%%
\begin{figure}[ht]
\centerline{\epsfig{file=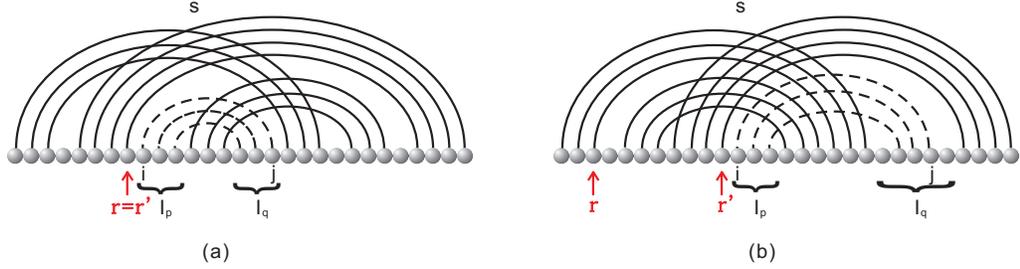,width=0.9\textwidth}\hskip8pt}
\caption{\small Stack-insertion: if the origin of the inserted stack
$(i,j,\sigma)$ is smaller than that of its predecessor (a), then $r=r'$.
Paraphrasing the situation we can express this as  ``left-insertion''
freezes the index $r$. Accordingly, (b) showcases the ``right-insertion'',
with its induced shift of the indices $r\mapsto r'$, both indices are
drawn in red. }
\label{F:insert}
\end{figure}
%%%
%%%%%%%%%%%%%%%%%%%%%%%%%%%%%%%%%%%%%%%%%%%%%%%%%%%%%%%%%%%%%%%%%%%%%%%%
%%%
We refer to the above procedure as $(i,j,\sigma)$-insertion and formally
express it via
\begin{equation}
(S,r)\Rightarrow_{(i,j,\sigma)} (S',r').
\end{equation}
Given a pair $(S,r)$ subsequent insertions induce a directed graph,
$G_{(S,r)}$, whose vertices are pairs $(S',r')$ and whose (directed)
arcs are given by
\begin{equation}
((S,r),(S',r')),\quad \text{\rm where}\quad
(S,r)\Rightarrow_{(i,j,\sigma)} (S',r').
\end{equation}
\begin{remark}
Note that the algorithm checks whether $(i,j,\sigma)$ can be added,
i.e.~(1) the bases $\{i,i+1,\cdots,i+\sigma-1,j-\sigma+1,\cdots,j-1,j\}$
are indeed unpaired and (2) $(i-1,j+1)$ is not a base pair.
The second property guarantees that the core of the stack $(i,j,\sigma)$
is an arc in the core of $S'$.
\end{remark}
We proceed by showing that $G_{(S,r)}$ is in fact a tree. In other words, the
insertion-procedure is an unambiguous grammar.
%%%
%%%%%%%%%%%%%%%%%%%%%%%%%%%%%%%%%%%%%%%%%%%%%%%%%%%%%%%%%%%%%%%%%%%%%%%%%%%%%%
%%%
\begin{proposition}\label{P:tree}
Let $T_1=\{S\mid \exists\, r;\, (S,r)\in T\}$ and $S_0$ be a 
$3$-noncrossing skeleton.\\
{\rm (a)} $G_{(S_0,r_0)}$ is a tree and for any two different vertices 
         $(S'_1,r'_1)$ and $(S'_2,r'_2)$
in  $G_{(S,r_0)}$, we have $S'_1\neq S'_2$.\\ 
{\rm (b)} For $k>3$, the graph morphism $\pi\colon \mathbb{T}\longrightarrow \mathbb{T}_1$, 
          given by $\pi((S,r))=S$ is not bijective.
\end{proposition}
%%%
%%%%%%%%%%%%%%%%%%%%%%%%%%%%%%%%%%%%%%%%%%%%%%%%%%%%%%%%%%%%%%%%%%%%%%%%%%%%%%
%%%
\begin{remark}
For any $k>3$, $G_{(S_0,r_0)}$ is a tree. However assertion (b)
indicates that it is {\it really} a tree of pairs. That means, 
stack-insertions will in general generate two different pairs with 
equal first coordinate.
\end{remark}

\begin{proof}
We prove assertion (a) by induction on the number of inserted 
arcs, $\ell$.
For $\ell=0$ there is nothing to prove. For $\ell=1$, the pairs
$(S,r_0)$ and $(S',r')$ differ by exactly one stack, $(i,j,\sigma)$,
whence the assertion.
%%%%%%%%%%%%%%%%%%%%%%%%%%%%%%%%%%%%%%%%%%%%%%%%%%%%%%%%%%%%%%%%%%%%%%%%%%%%%
Our objective is now to show that for any two $(S_1',r_1')$ and $(S'_2,r'_2)$ 
obtained from the root $(S,r_0)$  via $\ell$ insertions, $S_1'\neq S_2'$ holds.
Suppose there exists some $(\tilde{S},\tilde{r})$, such that
\begin{equation}
\diagram
& (\tilde{S},\tilde{r}) \dlto_{\rm inertion}\drto^{\rm
insertion}& \\
(S'_1,r'_1) & & (S'_2,r'_2) 
\enddiagram
\end{equation}
If the inserted stacks coincide, we have $(S_1',r_1') = (S_2',r_2')$
and there is nothing to prove. Otherwise, we obtain $S'_1\ne S'_2$, 
which implies $(S_1',r_1') \ne (S_2',r_2')$, whence (a).
Suppose next, we have the following situation
\begin{equation}
\diagram
 & &(S_0,r_0)\dlto_{\text{\rm unique
 path}}\drto^{\text{\rm unique path}}
& \\
 & (S_1,r_1) \dto_{\text{\rm insertion}} & & (S_2,r_2)
\dto^{\text{\rm insertion}}\\
 & (S'_1,r'_1) &  & (S'_2,r'_2)\\
\enddiagram
\end{equation}
where the uniqueness of the paths ending at $(S_1,r_1)$ and $(S_2,r_2)$
is guaranteed by the induction hypothesis.
By assumption we have $(S_1,r_1)\ne (S_2,r_2)$ and $S_1$ and $S'_1$ as well 
as $S_2$ and $S'_2$ differ by exactly one stack. 
Again by induction hypothesis, we have $S_1\ne S_2$, whence
\begin{equation}
(S_1,r_1)\Rightarrow_{\alpha=(i_\alpha,j_\alpha,\sigma_\alpha)} (S_1',r_1'),\
(S_2,r_2)\Rightarrow_{\beta=(i_\beta,j_\beta,\sigma_\beta)} (S_2',r_2') \quad
\text{\rm and}\quad
 S_1\neq S_2.
\end{equation}
We now prove the inductive step by contradition. Suppose we have $S'_1=S'_2$,
then we can conclude that $\alpha\neq \beta$ and there exists some 
$(\tilde{S},\tilde{r})$ such that
\begin{equation}\label{E:murx1}
\diagram
& & (S,r_0)\dto^{\text{\rm unique path}} & \\
& &(\tilde{S},\tilde{r})\dlto_{\beta}\drto^{\alpha} & \\
& (S_1,r_1) \dto_{\alpha} & & (S_2,r_2) \dto^{\beta} \\
&(S'_1,r'_1) & & (S'_2,r'_2)
\enddiagram
\end{equation}
Indeed, we define $\tilde{S}$ to be the skeleton derived from
$(S_0,r_0)$ by inserting all $S'_1$-arcs except of 
$\alpha,\beta$. It is
clear that the skeleton $\tilde{S}$ exists since its stack-set is a subset of 
the stack-set of $S'_1$. By construction, $\tilde{S}$
differs from $S_1$ and $S_2$ via the stacks $\alpha$ and $\beta$,
respectively. By induction hypothesis, there exists a unique path
from $(S,r_0)$ to $(\tilde{S},\tilde{r})$, which implies the
existence of a unique $\tilde{r}$. Furthermore, by induction
hypothesis, the paths from $(S_0,r_0)$ to $(S_1,r_1)$ and $(S_2,r_2)$ are 
unique and consequently contain $(\tilde{S},\tilde{r})$, whence we have 
the situation given in eq.~(\ref{E:murx1}).\\
As $\alpha$ and $\beta$ are both minimal, without loss of generality we 
may assume $i_\alpha< i_\beta$.
Let us consider the insertion-path $(\tilde{S},\tilde{r})\Rightarrow_\beta 
(S_1,r_1) \Rightarrow_\alpha (S_1',r_1')$. According to this insertion, we 
obtain $r_1<i_\alpha$ and by construction $[r_1+1,i_\beta-1]$ is an 
$S_1$-interval.
If $j_\alpha<i_\beta$, then $\alpha$ does not cross any arcs in $S'_1$, 
which is impossible. If $j_\alpha>j_\beta$, we arrive at $\beta\prec 
\alpha$, which contradicts minimality of $\alpha$. 
Therefore, we have  $i_\beta<j_\alpha<j_\beta$, i.e.~the arcs $\alpha$ and
$\beta$ are crossing.
Next we consider $(\tilde{S},\tilde{r})\Rightarrow_\alpha
(S_2,r_2) \Rightarrow_\beta (S_2',r_2')$. Accordingly, $\alpha$ must be 
crossed by some $(\tilde{S},\tilde{r})$-stack, say $\gamma=(i_\gamma, 
j_\gamma,\sigma_\gamma)$. We next put $\gamma$ into the context of the 
insertion-path
$(\tilde{S},\tilde{r})\Rightarrow_\beta (S_1,r_1) \Rightarrow_\alpha 
(S_1',r_1')$ and observe that $\gamma$ necessarily crosses $\beta$. 
Indeed, otherwise we have the following three scenarios: 
$i_\gamma>j_\beta$, $j_\gamma\le r_1$ or $i_\gamma\le r_1, j_\gamma>j_\beta$. 
In all three cases $\gamma$ cannot cross $\alpha$ since $i_\gamma,
j_\gamma\not\in [r_1+1,i_\beta-1]$, see Fig.\ref{F:proveofske}.
%%%
%%%%%%%%%%%%%%%%%%%%%%%%%%%%%%%%%%%%%%%%%%%%%%%%%%%%%%%%%%%%%%%%%%%%%%%%
%%%
\begin{figure}[ht]
\centerline{\epsfig{file=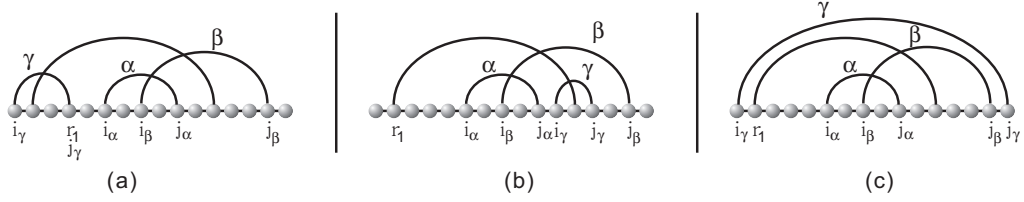,width=0.9\textwidth}\hskip8pt}
\caption{\small Illustration of the proof of Proposition~\ref{P:tree}.
The three different scenarios for a noncrossing $\gamma$,
representing stacks by isolated arcs.
(a) $j_\gamma\le r_1$, (b) $i_\gamma>j_\beta$ and (c) $i_\gamma\le r_1$,
$j_\gamma>j_\beta$.  }
\label{F:proveofske}
\end{figure}
%%%
%%%%%%%%%%%%%%%%%%%%%%%%%%%%%%%%%%%%%%%%%%%%%%%%%%%%%%%%%%%%%%%%%%%%%%%%
%%%
As a result, $\gamma$ necessarily crosses both stacks: $\alpha$ and $\beta$, 
which is a contradition to the fact that $S'_1$ is a $3$-noncrossing 
skeleton, whence $S'_1\ne S'_2$. In particular we obtain
$(S'_1,r'_1)\ne (S'_2,r'_2)$, the insertion path is
unique and $G_{(S,r_0)}$ is a tree. \\
In order to prove (b) we provide via Fig.\ref{F:ske-4non} an example, where
the implication $(S_1,r_1)\neq (S_2,r_2)\,\Rightarrow 
S_1\neq S_2$ does not hold. Note that $\mathbb{T}_{(S_0,r_0)}$ is still a 
tree.
\end{proof}

Next we prove that our unambiguous grammar indeed generates any
skeleton, which contains a given irreducible shadow.
%%%
%%%%%%%%%%%%%%%%%%%%%%%%%%%%%%%%%%%%%%%%%%%%%%%%%%%%%%%%%%%%%%%%%%%%%%
%%%
\begin {proposition}
Suppose we are given an irreducible shadow $S_0=IS_{i,j}$.
Let $\mathbb{T}(S_0)=G_{(S_0,0)}$ denote ist skeleton-tree and
let $\mathbb{S}(S_0)$ be the set of all skeleta, that contain $S_0$.
Then we have
\begin{equation}
\mathbb{T}(S_0)=\mathbb{S}(S_0).
\end{equation}
\end {proposition}
%%%
%%%%%%%%%%%%%%%%%%%%%%%%%%%%%%%%%%%%%%%%%%%%%%%%%%%%%%%%%%%%%%%%%%%%%%
%%%
\begin{proof}
Let $\mathscr{A}_S$ denote the set of $S$-arcs.
Obviously, for any vertex $(S,r)\in \mathbb{T}(S_0)$, $S$ is
a $3$-noncrossing skeleton such that 
$\mathscr{A}_{S_0}\subseteq \mathscr{A}_S$, whence
$\mathbb{T}(S_0)\subseteq \mathbb{S}(S_0)$ holds.
For an arbitrary $3$-noncrossing skeleton $S$, let $
\mathscr{A}_S^{\text{\rm ne}}$
denote the set of all nested stacks in $S$. Since each arc is either maximal
or nested we have 
$\mathscr{A}_S=\mathscr{A}_{S_0}\dot\cup \mathscr{A}_S^{\text{\rm ne}}$.
Sorting $\mathscr{A}_S^{\text{\rm ne}}$ via
the linear ordering of their leftmost paired base, we obtain
the sequence $\Sigma=(\alpha_1, \alpha_2,\cdots, \alpha_n)$.
We choose the first element $\alpha_k\in \Sigma$ which is
intersecting $S_0$ (not necessarily $\alpha_1$). Then we have
\begin{equation}
(S_0,r_0)\hookrightarrow_{\alpha_k} (S_1,r_1)
\end{equation}
where, $S_1\in \mathbb{T}(S_0)$. We proceed inductively, setting 
$\mathscr{A}_S^{\text{\rm ne}}=
\mathscr{A}_S^{\text{\rm ne}}\setminus\alpha_k$ and
proceed inductively until $\mathscr{A}_S^{\text{\rm ne}}=\varnothing$.
By construction, each $S_k$ is in $\mathbb{T}(S_0)$, and $S_n=S$.
Accordingly, we constructed an insertion-path in $\mathbb{T}(S_0)$
from $S_0$ to $S$, from which $\mathbb{S}(S_0)
\subset \mathbb{T}(S_0)$ follows.
\end{proof}

%%%
%%%%%%%%%%%%%%%%%%%%%%%%%%%%%%%%%%%%%%%%%%%%%%%%%%%%%%%%%%%%%%%%%%%%%%%%
%%%
\begin{figure}[ht]
\centerline{\epsfig{file=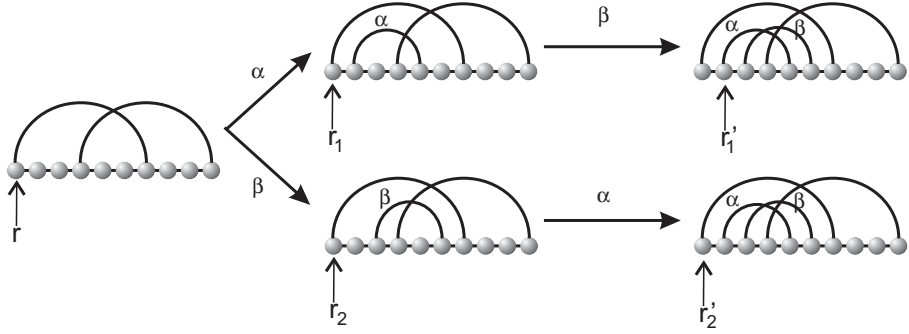,width=0.8\textwidth}\hskip8pt}
\caption{\small Illustration of assertion (b) of Proposition~\ref{P:tree}:
the case $k>3$. While $\mathbb{T}_{(S_0,r_0)}$ is still
a tree (over pairs), the implication $(S_1,r_1)\neq (S_2,r_2)\,\Rightarrow 
S_1\neq S_2$ does not hold in general.}
\label{F:ske-4non}
\end{figure}
%%%
%%%%%%%%%%%%%%%%%%%%%%%%%%%%%%%%%%%%%%%%%%%%%%%%%%%%%%%%%%%%%%%%%%%%%%%%
%%%

%%%
%%%%%%%%%%%%%%%%%%%%%%%%%%%%%%%%%%%%%%%%%%%%%%%%%%%%%%%%%%%%%%%%%%
%%%

\section{Phase III: Saturation}\label{S:sat}

%%%
%%%%%%%%%%%%%%%%%%%%%%%%%%%%%%%%%%%%%%%%%%%%%%%%%%%%%%%%%%%%%%%%%%%
%%%
In this section we discuss the third phase of {\sf cross}. 
The skeleta-trees constructed in the second phase organized the 
non-inductive substructures of an irreducible shadow derived in phase 
one. The objective of the saturation phase is to inductively ``fill'' 
the remaining intervals of a given skeleton with specific substructures. 
Basically, all routines employed here follow the DP-paradigm. 
However, we store a vector of structures rather than energies and 
implement context sensitive DP-routines.

Suppose we are given a skeleta-tree $\mathbb{T}(S_0)$ with root $S_0$.
Let the order of $S$, $\omega(S)$, denote the number of $\prec$-maximal 
$S$-arcs, see Fig.\ref{F:order}. 
Furthermore, let $\Sigma_{i,j}$ and $\Sigma^{[r]}_{i,j}$ be some 
subset of structures over $\{i,i+1,\dots,j-1,j\}$ and those of order 
$r$, respectively.
%%%%
%%%%%%%%%%%%%%%%%%%%%%%%%%%%%%%%%%%%%%%%%%%%%%%%%%%%%%%%%%%%%%%%%%%%%%
%%%
\begin{figure}[ht]
\centerline{\epsfig{file=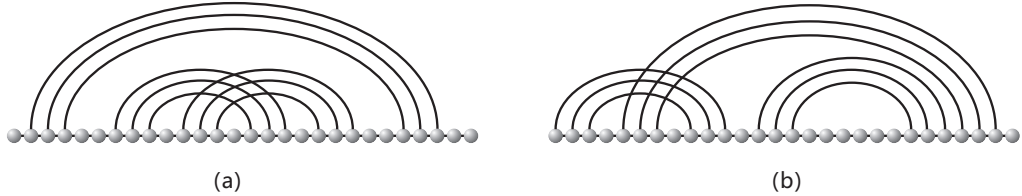,width=0.9\textwidth}\hskip8pt}
\caption{\small Order: In (a) we display a structure of order one. 
(b) showcases a structure of order two.}
\label{F:order}
\end{figure}
%%%%
%%%%%%%%%%%%%%%%%%%%%%%%%%%%%%%%%%%%%%%%%%%%%%%%%%%%%%%%%%%%%%%%%%%%%%
%%%
Let $\mathbb{M}_{i,j}$ denote the set of saturated skeleta over 
$\{i,i+1,\dots,j-1,j\}$ and $OSM(i,j)\in \mathbb{M}_{i,j}$
be a mfe-saturated skeleton. Furthermore, let $OS(i,j)$ be a mfe-structure, 
which is a union of disjoint $OSM(i_1,j_1), \dots OSM(i_r,j_r)$ 
and unpaired nucleotides. By $OSM^{[x]}(i,j)$ and $OS^{[x]}(i,j)$ we denote 
the respective $OSM$ and $OS$ structures of order $x$.
%%%%
%%%%%%%%%%%%%%%%%%%%%%%%%%%%%%%%%%%%%%%%%%%%%%%%%%%%%%%%%%%%%%%%%%%%%%
%%%
\begin{figure}[ht]
\centerline{\epsfig{file=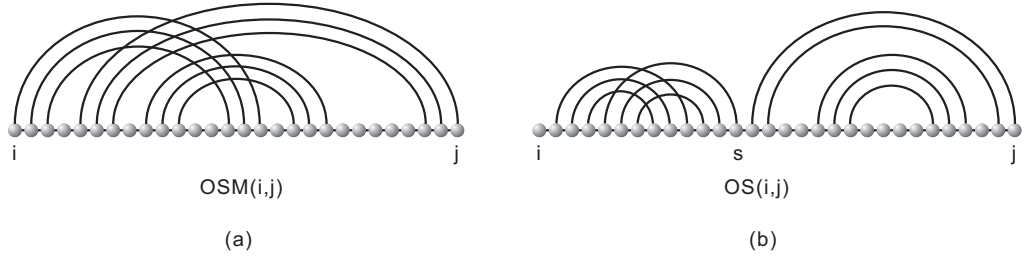,width=0.9\textwidth}\hskip8pt}
\caption{\small $OS$ vers.~$OSM$: we display a $OSM(i,j)$ (a), and
a $OS(i,j)$ structure (b). The $OS(i,j)$ structure shown in (b)
is evidently an union of of the structures $OSM(i,s)$ and $OSM(s+1,j)$ and
the unpaired nucleotide at position $i$. }
\label{F:ExOSOSM}
\end{figure}
%%%%
%%%%%%%%%%%%%%%%%%%%%%%%%%%%%%%%%%%%%%%%%%%%%%%%%%%%%%%%%%%%%%%%%%%%%%
%%%
In order to describe the context-sensitive saturation procedure in 
{\sf cross} we denote by $OS_{\text{\rm mul}}(i,j)$, $OS_{\text{\rm pk}}
(i,j)$ and $OS_0(i,j)$, the mfe-structures nested in a multi-loop,
pseudoknot and otherwise, respectively.

For a given a skeleton $S_{i,j}$, we specify the mapping
$S_{i,j}\mapsto OSM(S_{i,j})$ as follows: suppose $S_{i,j}$ has
$n_1$ intervals, $I_1,\dots,I_{n_1}$ labelled from left to right.
For given interval $I_r=[i_r,j_r]$ and $s_r\in\Sigma_{i_r,j_r}$ we
consider the insertion of $s_r$ into $I_r$, distinguishing the
following
four cases:\\
{\sf Case(1)}. $I_r$ is contained in a hairpin-loop.\\
\underline{$\omega(s_r)=0$}. That is we have $s_r=\varnothing$. The loop
generated by the $s_r$-insertion remains obviously a hairpin-loop, i.e.
$
((i_r-1,j_r+1),[i_r,j_r]),
$
with energy $H(i_r-1,j_r+1)$. \\
\underline{$\omega(s_r)=1$}.
Let $(p,q)$ be the unique, maximal $s_r$-arc. Then $s_r$-insertion produces
the interior-loop
$$
((i_r-1,j_r+1),[i_r,p-1],(p,q),[q+1,j_r]),
$$
with energy $I(i_k-1,j_k+1,p,q)$. Note that $p=i_r$ implies $q\neq j_r$ and
$s_r\in OSM_0^{[1]}(p,q)$. \\
\underline{$\omega(s_r)\ge 2$}.
In this case inserting $s_r$ into $I_r$ creates a multi-loop in which
$s_r$ is nested. Then $s_k\in OS^{[\ge 2]}_{\text{\rm mul}}$,
see Fig.\ref{F:OSM_1}.
Let $\epsilon(s)$ denote the energy of structure $s$. We select the set 
of all structures $s_r$ such that
$$
\epsilon(s_r)=\text{\rm min}
\begin{cases}
H(i_r-1,j_r+1)  \\
I(i_r-1,j_r+1,p,q) + \epsilon(OSM_0^{[1]}(p,q)) \\ 
\quad \quad  \forall i_r\le p<q\le j_r \ \text{\rm and } 
\ p=i_r, \Rightarrow q \neq
j_k \\
M+P_1+\epsilon(OS^{[\ge 2]}_{\text{\rm mul}}(i_r,j_r)).
\end{cases}
$$
Here, $M$ is the energy penalty for forming a multi-loop and
$P_1$ is the energy score of a closing-pair in multi-loop.
%%%%%%
%%%%%%%%%%%%%%%%%%%%%%%%%%%%%%%%%%%%%%%%%%%%%%%%%%%%%%%%%%%%%%%%%
%%%%%%%
\begin{figure}[ht]
\centerline{\epsfig{file=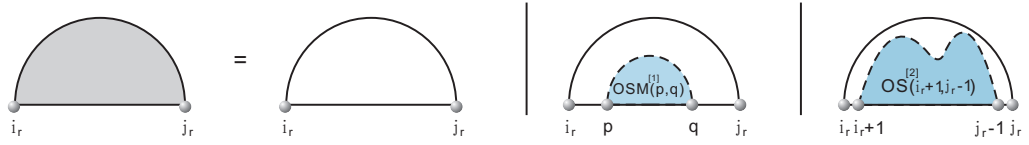,width=0.9\textwidth}\hskip8pt}
\caption{\small Saturation in hairpin-loops: the interval on the left 
hand side is filled with substructures $s_r$ such that $\omega(s_r)=0$ 
(left), $\omega(s_r)=1$ (middle) or $\omega(s_r)\ge 2$ (right).} 
\label{F:OSM_1}
\end{figure}
%%%%%%%%
%%%%%%%%%%%%%%%%%%%%%%%%%%%%%%%%%%%%%%%%%%%%%%%%%%%%%%%%%%%%%%%%%%%%
%%%%%%%%

%%%%%%%%%%%%%%%%%%%%%%%%%%%%%%%%%%%%%%%%%%%%%%%%%%%%%%%%%%%%%%%%%%%%
{\sf Case(2)}. $I_r$ is contained in a pseudoknot loop. \\
%%%%%%%%%%%%%%%%%%%%%%%%%%%%%%%%%%%%%%%%%%%%%%%%%%%%%%%%%%%%%%%%%%%%
\underline{$\omega(s_r)=0$}. That is we have $s_r=\{\varnothing \}$ and
the unpaired bases in $I_r$ are considered to be contained in a
pseudoknot. \\
\underline{$\omega(s_r)\ge 1$}. In this case, $s_r$ is a substructure
which is nested in a pseudoknot, see Fig.\ref{F:OSM_2}. As a result
our selection criterion is given by \\
$$
\epsilon(s_r)= \text{\rm min}
\begin{cases}
(j_r-i_r+1) \cdot Q_{\text{\rm pk}} \\
\epsilon(OS_{\text{\rm pk}}(i_r,j_r)).
\end{cases}
$$
where $(j_r-i_r+1)\in\mathbb{N}$ is the number of unpaired bases in
$I_r$, and $Q_{\text{\rm pk}}$ is the energy score of the unpaired bases in a
pseudoknot.
%%%%%
%%%%%%%%%%%%%%%%%%%%%%%%%%%%%%%%%%%%%%%%%%%%%%%%%%%%%%%%%%%%%%%%%%%
%%%%%
\begin{figure}[ht]
\centerline{\epsfig{file=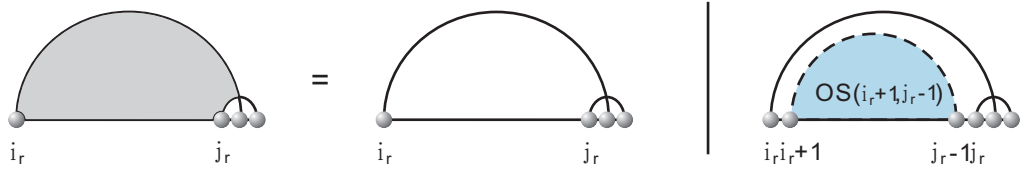,width=0.9\textwidth}\hskip8pt}
\caption{\small Saturation of interval nested in a pseudoknot.}
\label{F:OSM_2}
\end{figure}
%%%%%%
%%%%%%%%%%%%%%%%%%%%%%%%%%%%%%%%%%%%%%%%%%%%%%%%%%%%%%%%%%%%%%%%%%%
%%%%%%

{\sf Case(3)}. $I_r$ is contained in a multi-loop. In analogy to
case (2), we distinguish the following cases:\\
\underline{$\omega(s_r)=0$}. That is we have $s_r=\{\varnothing\}$.
The unpaired bases in $I_r$ are considered to be contained in a multi-loop. \\
\underline{$\omega(s_r)\ge 1$}. In this case, $s_r$ is a
substructure nested in a multi-loop, see Fig.\ref{F:OSM_3}. Accordingly, we
select all structures satisfying\\
\begin{equation*}
\epsilon(s_r) = \text{\rm min}
\begin{cases}
(j_r-i_r+1) \cdot Q_{\text{\rm mul}} \\
\epsilon(OS_{\text{\rm mul}}(i_r,j_r)),
\end{cases}
\end{equation*}
where $Q_{\text{\rm mul}}$ denotes the energy score of the unpaired bases in a
multi-loop.
%%%%%
%%%%%%%%%%%%%%%%%%%%%%%%%%%%%%%%%%%%%%%%%%%%%%%%%%%%%%%%%%%%%%%%%%%
%%%%%
\begin{figure}[ht]
\centerline{\epsfig{file=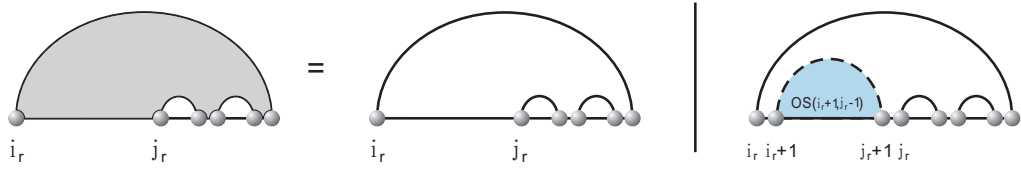,width=0.9\textwidth}\hskip8pt}
\caption{\small Saturation of an interval contained in a multi-loop. }
\label{F:OSM_3}
\end{figure}
%%%%%%
%%%%%%%%%%%%%%%%%%%%%%%%%%%%%%%%%%%%%%%%%%%%%%%%%%%%%%%%%%%%%%%%%%%
%%%%%%

{\sf Case(4)} $I_r$ is contained in an interior-loop. By
construction, the latter is formed by the pair $(I_r,I_l)$, where
$r<l$. We then select pairs $s_r$ in $\Sigma_{i_r,j_r}$ and $s_l$ in
$\Sigma_{i_l,j_l}$. Note that only the first coordinate of the pair
$(I_r,I_l)$ is considered.\\
\underline{$\omega(s_r)=0$ and $\omega(s_l)=0$}. Obviously, in this case 
the loop formed by $I_r$ and $I_l$ remains an interior-loop
$$
((i_r-1,j_l+1),[i_r,j_r],(j_r+1,i_l-1),[i_l,j_l]),
$$
whose energy is given by $I(i_r-1,j_l+1,j_r+1,i_l-1)$.\\
\underline{$\omega(s_r)\ge 1$ and $\omega(s_l)=0$}. In this case,
$s_l=\{\varnothing\}$. $I_r$ and $I_l$ create a multi-loop, in which
$s_r$ and the substructure $G_{j_r+1,i_l-1}$ are nested.\\
\underline{$\omega(s_r)=0$ and $\omega(s_l)\ge 1$}. Completely analogous 
to the previous case.\\
\underline{$\omega(s_r)\ge 1$ and $\omega(s_l)\ge 1$}. In this case,
$I_r$ and $I_l$ create a multi-loop, in which $s_r$, $s_l$ and
$G_{j_r+1,i_l-1}$ are nested, see Fig.\ref{F:OSM_4}.\\
Accordingly, we select all pairs of structures $(s_r,s_l)$ satisfying \\
$$
\epsilon(s_r)+\epsilon(s_l)= \text{\rm min}
\begin{cases}
I(i_r-1,j_l+1,j_r+1,i_l-1) \\
M+2P_1+\epsilon(OS_{\text{\rm mul}}(i_r,j_r))+(j_l-i_l+1)\cdot 
Q_{\text{\rm mul}} \\
M+2P_1+\epsilon(OS_{\text{\rm mul}}(i_l,j_l))+(j_k-i_k+1)\cdot 
Q_{\text{\rm mul}} \\
M+2P_1+\epsilon(OS_{\text{\rm mul}}(i_r,j_r))+
\epsilon(OS_{\text{\rm mul}}(i_l,j_l)) \\
\end{cases}
$$

%%%%%%
%%%%%%%%%%%%%%%%%%%%%%%%%%%%%%%%%%%%%%%%%%%%%%%%%%%%%%%%%%%%%%%%%%%
%%%%%%
\begin{figure}[ht]
\centerline{\epsfig{file=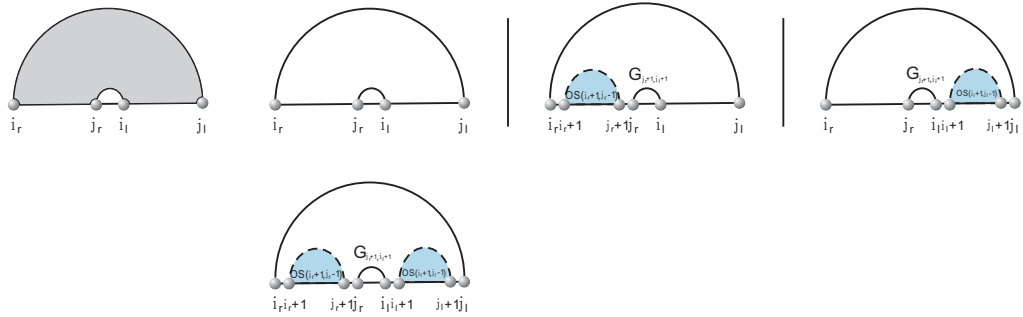,width=0.9\textwidth}\hskip8pt}
\caption{\small Saturation of an interval contained in an interior-loop,
which is obtained by $I_r$ and $I_l$, where $r<l$. } \label{F:OSM_4}
\end{figure}
%%%%
%%%%%%%%%%%%%%%%%%%%%%%%%%%%%%%%%%%%%%%%%%%%%%%%%%%%%%%%%%%%%%%%%%%%%%
%%%

Accordingly, we inductively saturate all intervals and in case of interior 
loops interval-pairs and thereby derive $OSM(S_{i,j})$. 
Then we select an energy-minimal $OSM(i,j)$ substructure
from the set of all $OSM(S_{i,j})$ for any skeleton $S_{i,j}$.

As for the construction of $OS(i,j)$ via $OSM(i',j')$, we consider
position $i$ in $OS(i,j)$. If $i$ is  paired, then $i$ is contained
in some $OSM(i,s)$.  Then $OS(i,j)$ induces a substructure $S_2$ over
$\{s+1,\dots,j\}$. By construction $OS(i,j)= OSM(i,s)\dot\cup S_2$, whence
$S_2=OS(s+1,j)$ and in particular we have
\begin{equation}
\epsilon(OS(i,j))=\epsilon(OSM(i,s))+\epsilon(OS(s+1,j)).
\end{equation}
Suppose next $i$ is
unpaired in $OS(i,j)$. Since $\epsilon$ is a loop-based energy, we
can conclude  $OS(i,j)=\{\varnothing\} \dot\cup OS(i+1,j)$, i.e.~we have
\begin{equation}
\epsilon(OS(i,j))=\epsilon(OS(i+1,j)) + Q
\end{equation}
where $Q$ represents the energy contribution of a single, unpaired
nucleotide.
Accordingly, we can inductively construct $OS(i,j)$ via the criterion
$$
\epsilon(OS(i,j))=\text{\rm min}\{\epsilon(OS(i+1,j))+Q,
\epsilon(OSM(i,s))+\epsilon(OS(s+1,j))\},\quad \forall i<s\leq j.
$$
%%%%
%%%%%%%%%%%%%%%%%%%%%%%%%%%%%%%%%%%%%%%%%%%%%%%%%%%%%%%%%%%%%%%%%%%%%%
%%%
\begin{figure}[ht]
\centerline{\epsfig{file=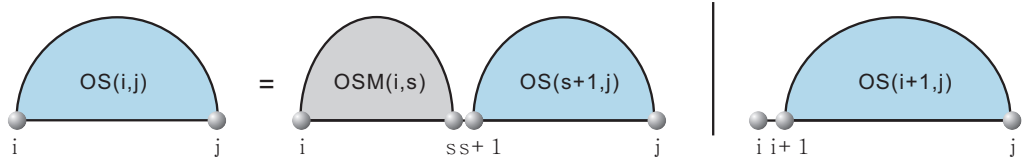,width=0.9\textwidth}\hskip8pt}
\caption{\small Constructing $OS(i,j)$: inductive decomposition of
the optimal structure, $OS(i,j)$, into saturated skeleta, $OSM(i,s)$
and unpaired nucleotides.} \label{F:OS}
\end{figure}
%%%%
%%%%%%%%%%%%%%%%%%%%%%%%%%%%%%%%%%%%%%%%%%%%%%%%%%%%%%%%%%%%%%%%%%%%%%
%%%

Now we can inductively construct the array of structures 
$OS(i,j)$ and $OSM(i,j)$ via $OS$ and $OSM$ structures over smaller 
intervals. As a result, we finally obtain the structure $OS(1,n)$, 
i.e.~the mfe-structure, see Fig.\ref{F:OSmatrix}.

%%%
%%%%%%%%%%%%%%%%%%%%%%%%%%%%%%%%%%%%%%%%%%%%%%%%%%%%%%%%%%%%%%%%%%%%%%%%%%%%%%
%%%
\begin{figure}[ht]
\centerline{\epsfig{file=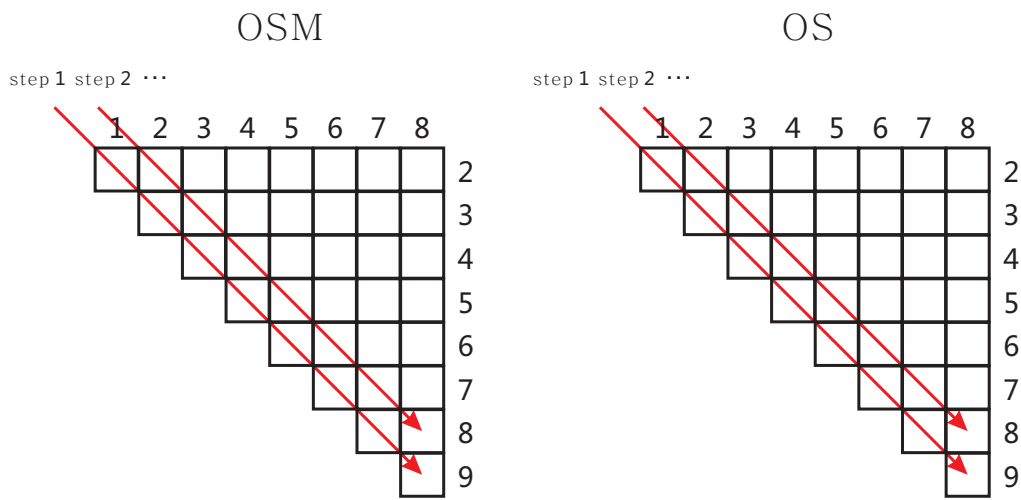,width=0.9\textwidth}\hskip8pt}
\caption{\small Inductive construction of $OS$ and $OSM$ structures: 
in the $s$-th step, 
we first construct $OSM(i,i+s)$, for any $0<i<n-s+1$. We then construct 
$OS(i,i+s)$ recruiting $OSM$-structures over intervals of lengths strictly
smaller than $s$.}
\label{F:OSmatrix}
\end{figure}
%%%
%%%%%%%%%%%%%%%%%%%%%%%%%%%%%%%%%%%%%%%%%%%%%%%%%%%%%%%%%%%%%%%%%%%%%%%%%%%%%%
%%%
%%%
%%%%%%%%%%%%%%%%%%%%%%%%%%%%%%%%%%%%%%%%%%%%%%%%%%%%%%%%%%%%%%%%%%%%%%%%%%%%%%
%%%

\section{Synopsis}

%%%
%%%%%%%%%%%%%%%%%%%%%%%%%%%%%%%%%%%%%%%%%%%%%%%%%%%%%%%%%%%%%%%%%%%%%%%%%%%%%%
%%%

After providing the necessary background and context on pseudoknot folding 
routines and $k$-noncrossing structures, we discussed in detail in Sections 
\ref{S:motifs},\ref{S:skeleta} and \ref{S:sat} the three phases of 
{\sf cross}, see Fig.\ref{F:sketch}. Now, that the key ideas are presented, 
we proceed by integrating and discussing our results.
%%%
%%%%%%%%%%%%%%%%%%%%%%%%%%%%%%%%%%%%%%%%%%%%%%%%%%%%%%%%%%%%%%%%%%%%%%%%%%%%%%
%%%
\begin{figure}[ht]
\centerline{\epsfig{file=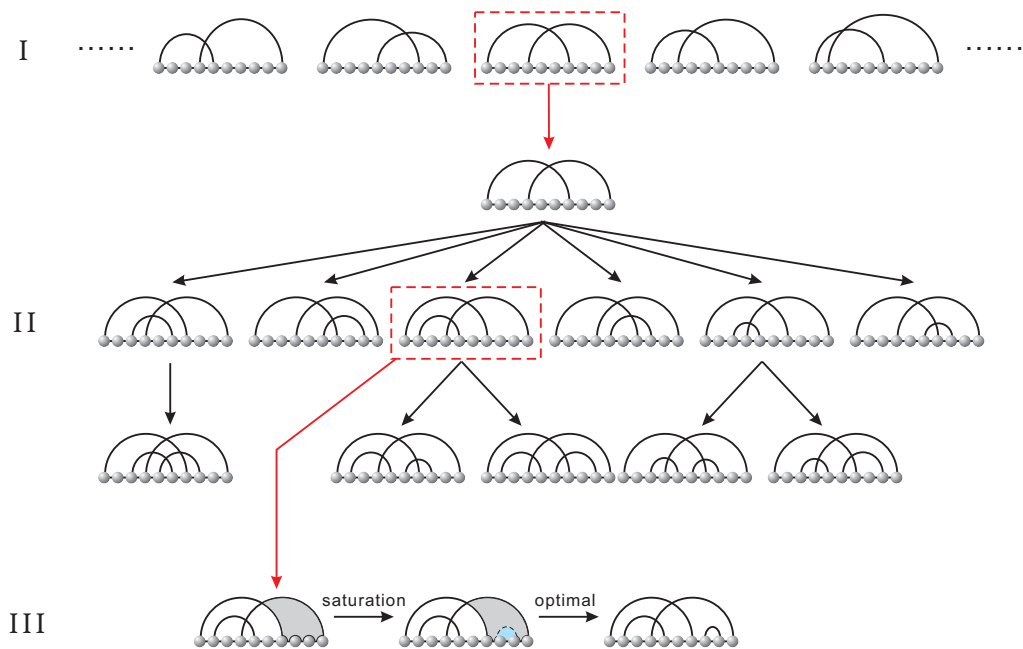,width=0.9\textwidth}\hskip8pt}
\caption{\small An outline of {\sf cross}: the generation of motifs (I), the 
construction of skeleta-trees, rooted in irreducible shadows (II) and 
the saturation (III), during which, via DP-routines, optimal fillings of 
skeleta-intervals are derived.
}
\label{F:sketch}
\end{figure}
%%%
%%%%%%%%%%%%%%%%%%%%%%%%%%%%%%%%%%%%%%%%%%%%%%%%%%%%%%%%%%%%%%%%%%%%%%%%%%%%%%
%%%
{\sf Cross} is an {\it ab initio} folding algorithms, 
which is guaranteed to 
search all $3$-noncrossing, $\sigma$-canonical structures and derives the 
corresponding loop-based mfe-configuration. 
A detailed description of the loop-energies as well as specific implementation 
particulars on how to generate the skeleta-trees of Section~\ref{S:skeleta} 
via a certain matrix construction can be found at
$$
{\tt www.combinatorics.cn/cbpc/cross.html}
$$
We remark that the code is improved and new features are being added, 
for instance, we currently work towards deriving the partition function 
version of {\sf cross}, the generalization for arbitrary $k$ and a 
fully parallel implementation.
The design of {\sf cross} is fundamentally different from that of the 
pseudoknot DP-routines found in the literature. Point in case being the 
algorithm of \cite{RE:98}, as outlined in Section~\ref{S:Introduction}. 
We showed that the latter cannot create any nonplanar $3$-noncrossing 
structure and furthermore cannot control the maximal number of mutually 
crossing arcs (crossing number). Consequently, DP-routines generate 
pseudoknot complexity by ``just'' increasing this very crossing number. 
The class of nonplanar $3$-noncrossing structures illustrates however, that 
structural complexity is not tantamount to the crossing number. 

One key difference to any other pseudoknot 
folding algorithm is the fact that {\sf cross} has a transparent, 
combinatorially specified, output class. This feature exists exclusively 
in secondary structure folding algorithms, 
where it is {\it by construction} implied. 
This specification is based on a novel
combinatorial class, the $k$-noncrossing RNA structures and their exact and
asymptotic enumeration \cite{Reidys:07pseu,07lego,Reidys:08ma}. 
The concept of $k$-noncrossing RNA structures is based on the
combinatorial work of Chen~{\it et al.} \cite{Chen, Reidys:08cross}. 
The implications of this framework are profound: for $k=3,4,\dots,6$ it 
is possible, employing central limit theorems for $k$-noncrossing structures 
\cite{Reidys:07limit,Reidys:08limit} to derive a variety of 
generic properties of sequence-structure maps into RNA pseudoknot structures,
irrespective of energy parameters \cite{Reidys:08loc,Reidys:bmc}.

Furthermore {\sf cross} is capable to generate novel classes of 
pseudoknots. Even in its current implementation, i.e.~restriced to
$3$-noncrossing structures it can generate any non-planar configuration.
As mentioned already, the extension of {\sf cross} to a version capable 
of folding any $k$-noncrossing structure, is work in progress. In this
context, assertion (b) of Proposition~\ref{P:tree} shows that novel 
constructions are required for efficient folding.
{\sf Cross} is {\it by design} an algorithm of exponential time complexity by 
virtue of its construction of its shadows and skeleta-trees. Only in its 
saturation phase it employs vector versions of DP-routines. Beyond the
asymptotic analysis of motifs, given in Section~\ref{S:motifs}, a detailed 
study of the performance of {\sf cross} is work in progress. It appears 
however, that the folding times of random sequences are exponentially 
distributed. In Fig.\ref{F:time} 
%%%
%%%%%%%%%%%%%%%%%%%%%%%%%%%%%%%%%%%%%%%%%%%%%%%%%%%%%%%%%%%%%%%%%%%%%%%%%%%%%%
%%%
\begin{figure}[ht]
\centerline{\epsfig{file=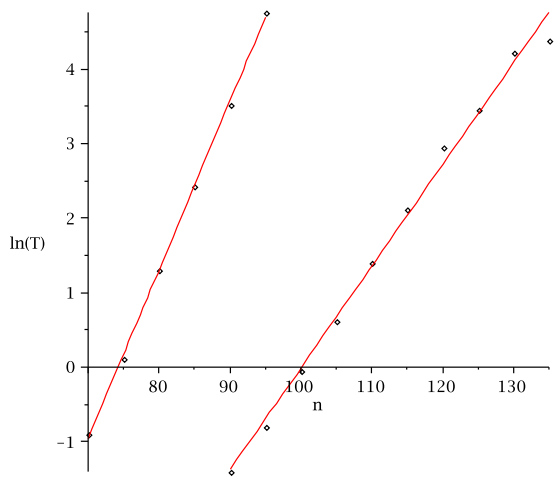,width=0.5\textwidth}\hskip8pt}
\caption{\small Mean folding times: we display the logarithm of the
folding times of $1000$ random sequences as a function of the sequence 
length. For $3$-canonical and $4$-canonical structures the linear fits are
given by $0.2263n-19.796$ (left) and $0.1364n-13.659$ (right),
respectively.
}
\label{F:time}
\end{figure}
we display the logarithm of the mean folding time of $1000$ random sequences. 
These data suggest exponential times with the exponential growth rates of 
$\approx 1.146$ and $\approx 1.254$, for $3$-canonical and $4$-canonical 
structures, respectively. 
In particular, a random sequence of length $100$ folded via a single core, 
$2.2$-GHz CPU exhibits a mean folding time of $279$ seconds with standard
deviation of $267744$ seconds.

%%%
%%%%%%%%%%%%%%%%%%%%%%%%%%%%%%%%%%%%%%%%%%%%%%%%%%%%%%%%%%%%%%%%%%%%%%%%
%%%
{\bf Acknowledgments.}
%%%
%%%%%%%%%%%%%%%%%%%%%%%%%%%%%%%%%%%%%%%%%%%%%%%%%%%%%%%%%%%%%%%%%%%%%%%%%%
%%%
%We are grateful to ??? for helpful discussions.
This work
was supported by the 973 Project, the PCSIRT Project of the Ministry
of Education, the Ministry of Science and Technology, and the
National Science Foundation of China.
%%%
%%%%%%%%%%%%%%%%%%%%%%%%%%%%%%%%%%%%%%%%%%%%%%%%%%%%%%%%%%%%%%%%%%%%%%%%
%%%
{
 \bibliography{algo_final_article.bib} 

\begin{thebibliography}{10}

\bibitem{W:HDV}
The \textsc{HDV} structure in nature.
\newblock http://132.229.50.4/~batenburg/PKBase/PKB00075.html.

\bibitem{Science:05a}
Mapping \textsc{RNA} form and function.
\newblock {\em Science}, 2, 2005.

\bibitem{Akutsu:00a}
T.~Akutsu.
\newblock Dynamic programming algorithms for \textsc{RNA} secondary prediction
  with pseudoknots.
\newblock {\em Discr. Appl. Math.}, 104:45--62, 2000.

\bibitem{Cao:06}
S.~Cao and S.~J. Chen.
\newblock Predicting \textsc{RNA} pseudoknot folding thermodynamics.
\newblock {\em Nucl. Acids. Res.}, 34(9):2634--2652, 2006.

\bibitem{MWM:95}
R.~Cary and G.~Stormo.
\newblock Graph-theoretic approach to \textsc{RNA} modeling using comparative
  data.
\newblock {\em Proc. Int. Conf. Intell. Syst. Mol. Biol.}, 3:75--80, 1995.

\bibitem{Chen}
W.~Y.~C. Chen, E.~Y.~P. Deng, R.~R.~X. Du, R.~P. Stanley, and C.~H. Yan.
\newblock Crossings and nestings of matchings and partitions.
\newblock {\em Trans. Am. Math. Soc.}, 359:1555--1575, 2007.

\bibitem{Reidys:08cross}
W.~Y.~C. Chen, J.~Qin, and C.~M. Reidys.
\newblock Crossing and nesting in tangled-diagrams.
\newblock {\em Elec. J. Comb.}, 15, 2008.

\bibitem{DeLisi:1971}
C.~DeLisi and D.~M. Crothers.
\newblock Prediction of \textsc{RNA} secondary structure.
\newblock {\em Proc. Natl. Acad. Sci, USA}, 68:2682--2685, 1971.

\bibitem{Dirks:04}
R.~M. Dirks and N.~A. Pierce.
\newblock An algorithm for computing nucleic acid base-pairing probabilities
  including pseudoknots.
\newblock {\em J. Comput. Chem.}, 25:1295--1304, 2004.

\bibitem{Eddy:04}
S.~R. Eddy.
\newblock How do \textsc{RNA} folding algorithms work?
\newblock {\em Nature Biotechnology}, 22:1457--1458, 2004.

\bibitem{MWM:65}
J.~Edmonds.
\newblock Maximum matching and polyhedron with $0,1$-vertices.
\newblock {\em J. Res. Nat. Bur. Stand.}, 69B:125--130, 1965.

\bibitem{Jaces:1960}
J.~R. Fresco, B.~M. Alberts, and P.~Doty.
\newblock Some molecular details of the secondary structure of ribonucleic
  acid.
\newblock {\em Nature}, 188:98--101, 1960.

\bibitem{MWM:76}
H.~N. Gabow.
\newblock An efficient implementation of \text{Edmonds'} algorithm for maximum
  matching on graphs.
\newblock {\em J. Asc. Com. Mach.}, 23:221--234, 1976.

\bibitem{Gessel:}
I.~Gessel and D.~Zeilberger.
\newblock Random walk in a weyl chamber.
\newblock {\em Proc. Amer. Math. Soc.}, 115:27--31, 1992.

\bibitem{Schuster:96}
W.~Gruener, R.~Giegerich, D.~Strothmann, C.~M. Reidys, Weber J., I.~L.
  Hofacker, P.~F. Stadler, and Schuster P.
\newblock Analysis of \textsc{RNA} sequence structure maps by exhaustive
  enumeration i. neutral networks.
\newblock {\em Monatsh. Chem.}, 127:375--389, 1996.

\bibitem{Reidys:96}
W.~Gruener, R.~Giegerich, D.~Strothmann, C.~M. Reidys, Weber J., I.~L.
  Hofacker, P.~F. Stadler, and Schuster P.
\newblock Analysis of \textsc{RNA} sequence structure maps by exhaustive
  enumeration. ii.
\newblock {\em Monatsh. Chem.}, 127:355--374, 1996.

\bibitem{Vienna:Server}
I.~L. Hofacker.
\newblock Vienna \textsc{RNA} secondary structure server.
\newblock {\em Nucl. Acids. Res.}, 31(13):3429--3431, 2003.

\bibitem{Stadler:98}
I.~L. Hofacker, M.~Fekete, C.~Flamm, M.~A. Huynen, S.~Rauscher, P.~E. Stolorz,
  and P.~F. Stadler.
\newblock Automatic detection of conserved \textsc{RNA} structure elements in
  complete \textsc{RNA} virus genomes.
\newblock {\em Nucl. Acids. Res.}, 26:3825--2836, 1998.

\bibitem{ViennaRNA}
I.~L. Hofacker, W.~Fontana, P.~F. Stadler, L.~S. Bonhoeffer, M.~Tacker, and
  P.~Schuster.
\newblock Fast folding and comparison of \textsc{RNA} secondary structures.
\newblock {\em Monatsh. Chem.}, 125:167--188, 1994.

\bibitem{Waterman:80}
J.~A. Howell, T.~F. Smith, and M.~S. Waterman.
\newblock Computation of generating functions for biological molecules.
\newblock {\em J. Appl. Math.}, 39:119--133, 1980.

\bibitem{Reidys:bmc}
F.~W.~D. Huang, L.~Y.~M. Li, and C.~M. Reidys.
\newblock Sequence-structure relations of pseudoknot \textsc{RNA}.
\newblock {\em Bioinformatics}.
\newblock in press.

\bibitem{Reidys:08limit}
F.~W.~D. Huang and C.~M. Reidys.
\newblock Statistics of canonical \textsc{RNA} pseudoknot structures.
\newblock {\em J. Theor. Biol.}
\newblock in press.

\bibitem{Stadler:rug}
M.~Huynen, P.~F. Stadler, and W.~Fontana.
\newblock Smoothness within ruggedness: the role of neutrality in adaptation.
\newblock {\em Proc. Natl. Acad. Sci, USA}, 93:397--401, 1996.

\bibitem{Reidys:07pseu}
E.~Y. Jin, J.~Qin, and C.~M. Reidys.
\newblock Combinatorics of \textsc{RNA} structures with pseudoknots.
\newblock {\em Bull. Math. Biol.}, 70(1):45--67, 2008.

\bibitem{07lego}
E.~Y. Jin and C.~M. Reidys.
\newblock \textsc{RNA}-lego: Combinatorial design of pseudoknot \textsc{RNA}.
\newblock {\em Adv. Appl. Math.}
\newblock in press.

\bibitem{Reidys:07limit}
E.~Y. Jin and C.~M. Reidys.
\newblock Central and local limit theorems for \textsc{RNA} structures.
\newblock {\em J. Theor. Biol.}, 250(3):547--559, 2008.

\bibitem{Reidys:08wang}
E.~Y. Jin, C.~M. Reidys, and R.~R. Wang.
\newblock Asympotic enumeration of $k$-noncrossing matchings.
\newblock Submitted.

\bibitem{Tinoco:1971}
I.~T. Jun, O.~C. Uhlenbeck, and M.~D. Levine.
\newblock Estimation of secondary structure in ribonucleic acids.
\newblock {\em Nature}, 230:362 -- 367, 1971.

\bibitem{Konings:95a}
D.~A.~M. Konings and R.~R. Gutell.
\newblock A comparison of thermodynamic foldings with comparatively derived
  structures of 16s and 16s-like r\textsc{RNA}s.
\newblock {\em RNA}, 1:559--574, 1995.

\bibitem{Loria:96a}
A.~Loria and T.~Pan.
\newblock Domain structure of the ribozyme from eubacterial ribonuclease.
\newblock {\em RNA}, 2:551--563, 1996.

\bibitem{Lyngso}
R.~B. Lyngs\o{} and C.~N.~S. Pedersen.
\newblock \textsc{RNA} pseudoknot prediction in energy-based models.
\newblock {\em J. Comput. Biol.}, 7:409--427, 2000.

\bibitem{Reidys:08ma}
G.~Ma and C.~M. Reidys.
\newblock Canonical \textsc{RNA} pseudoknot structures.
\newblock {\em J. Comput. Biol.}
\newblock in press.

\bibitem{Nebel:06}
D.~Metzler and M.~E. Nebel.
\newblock Predicting \textsc{RNA} secondary structures with pseudoknots by mcmc
  sampling.
\newblock {\em J. Math. Biol.}, 56(1-2):161--181, 2008.

\bibitem{Nussinov:1980}
R.~Nussinov and A.~B. Jacobson.
\newblock Fast algorithm for predicting the secondary structure of
  single-stranded \textsc{RNA}.
\newblock {\em Proc. Natl. Acad. Sci, USA}, 77:6309--6313, 1980.

\bibitem{Reidys:frame}
J.~Qin and C.~M. Reidys.
\newblock A combinatorial framework for \textsc{RNA} tertiary interaction.
\newblock 2007.
\newblock Submitted.

\bibitem{Reeder:04}
J.~Reeder and Giegerich. R.
\newblock Design, implementation and evaluation of a practical pseudoknot
  folding algorithm based on thermodynamics.
\newblock {\em Bioinformatics}, 5(104), 2004.

\bibitem{Reidys:08loc}
C.~M. Reidys.
\newblock Local connectivity of neutral networks.
\newblock {\em Bull. Math. Biol.}

\bibitem{Reidys:2002}
P.~F. Reidys, C. M.~andStadler.
\newblock Combinatorial landscapes.
\newblock {\em SIAM Review}, 44:3--54, 2002.

\bibitem{Ren:05}
J.~Ren, B.~Rastegari, A.~Condon, and H.~Hoos.
\newblock Hotkonts: Heuristic prediction of \textsc{RNA} secondary structures
  including pseudoknots.
\newblock {\em RNA}, 11:1494--1504, 2005.

\bibitem{RE:98}
E.~Rivas and S.~R. Eddy.
\newblock A dynamic programming algorithm for \textsc{RNA} structure prediction
  including pseudoknots.
\newblock {\em J. Mol. Biol.}, 285(5):2053--2068, 1999.

\bibitem{RE:00}
E.~Rivas and S.~R. Eddy.
\newblock The language of \textsc{RNA}: A formal grammar that includes
  pseudoknots.
\newblock {\em Bioinformatics}, 16:326--333, 2000.

\bibitem{ILM}
J.~Ruan, G.~Stormo, and W.~Zhang.
\newblock An iterated loop matching approch to the prediction.
\newblock {\em Bioinformatics}, 20:58--66, 2004.

\bibitem{Fontana:99}
P.~Schuster and W.~Fontana.
\newblock Chance and necessity in evolution: Lessons from \textsc{RNA}.
\newblock {\em Physica. D.}, 133:427--452, 1999.

\bibitem{Schuster:94}
P.~Schuster, W.~Fontana, P.~F. Stadler, and I.~L. Hofacker.
\newblock From sequences to shapes and back: A case study in \textsc{RNA}
  secondary structures.
\newblock {\em Proc. Roy. Soc. Lond. B}, 255:279--284, 1994.

\bibitem{Searls:02}
D.~B. Searls.
\newblock The language of genes.
\newblock {\em Nature}, 420:211–217, 2002.

\bibitem{Waterman:78a}
T.~F. Smith and M.~S. Waterman.
\newblock \textsc{RNA} secondary structure.
\newblock {\em Math. Biol.}, 42:31--49, 1978.

\bibitem{MWM:98}
J.~Tabaska, R.~Cary, H.~Gabow, and G.~Stormo.
\newblock An \textsc{RNA} folding method capable of identifying pseudoknots and
  base triples.
\newblock {\em Bioinformatics}, 14:691--699, 1998.

\bibitem{algo-independent}
M.~Tacker, P.~F. Stadler, E.~G. Bornberg-Bauer, Schuster P., I.~L. Hofacker,
  and P.~Schuster.
\newblock Algorithm independent properties of \textsc{RNA} secondary structure
  predictions.
\newblock {\em Europ. Biophy. J.}, 25:115--130, 1996.

\bibitem{Tinoco:73}
I.~Tinoco, P.~N. Borer, B.~Dengler, M.~D. Levine, O.~C. Uhlenbeck, D.~M.
  Crothers, and J.~Gralla.
\newblock Improved estimation of secondary structure in ribonucleic acids.
\newblock {\em Nature New Biology}, 246:40--41, 1973.

\bibitem{Uemura:99a}
Y.~Uemura, A.~Hasegawa, S.~Kobayashi, and T.~Yokomori.
\newblock Tree adjoining grammars for \textsc{RNA} structure prediction.
\newblock {\em Theor. Comput. Sci.}, 210:277--303, 1999.

\bibitem{Waterman:79a}
M.~S. Waterman.
\newblock Combinatorics of \textsc{RNA} hairpins and cloverleaves.
\newblock {\em Stud. Appl. Math.}, 60:91--96, 1979.

\bibitem{Waterman:86}
M.~S. Waterman and T.~F. Smith.
\newblock Rapid dynamic programming methods for \textsc{RNA} secondary
  structure.
\newblock {\em Adv. Appl. Math.}, 7:455--464, 1986.

\bibitem{Westhof:92a}
E.~Westhof and L.~Jaeger.
\newblock \textsc{RNA} pseudoknots.
\newblock {\em Curr. Opin. Struct. Biol.}, 2:327--333, 1992.

\bibitem{Zuker:1981}
M.~Zuker and P.~Stiegler.
\newblock Optimal computer folding of large \textsc{RNA} sequences using
  thermodynamics and auxiliary information.
\newblock {\em Nucl. Acids. Res.}, 9:133–148, 1981.

\end{thebibliography}
 \bibliographystyle{plain}  % Style BST file
} 

\end{document}